\documentclass[11pt]{article}

\usepackage[utf8]{inputenc}
\usepackage[english]{babel}
\usepackage{amsmath}
\usepackage{nicefrac}
\usepackage{amsthm}
\usepackage{amsfonts}
\usepackage{amssymb}
\usepackage{verbatim}
\usepackage{color}
\usepackage[arrow, matrix, curve]{xy}
\usepackage{bbm}
\usepackage[numbers]{natbib}
\usepackage{stmaryrd}
\usepackage{amssymb}
\usepackage{mathrsfs}

\newtheoremstyle{WreschTheoremstyle} 
                        {1.5em}    
                        {2.5em}    
                        {}         
                        {}         
                        {\bfseries}
                        {}        
                        {\newline} 
                        {\raisebox{0.6em}{\thmname{#1}\thmnumber{#2}\thmnote{ (#3)}}}
                        
\newcommand{\vertiii}[1]{{\left\vert\kern-0.25ex\left\vert\kern-0.25ex\left\vert #1 
    \right\vert\kern-0.25ex\right\vert\kern-0.25ex\right\vert}}

\newcommand{\R}{\mathbb{R}}
\newcommand{\C}{\mathbb{C}}
\newcommand{\N}{\mathbb{N}}

\newcommand{\K}{\mathcal{K}}
\newcommand{\B}{\mathbb{B}}
\newcommand{\E}{\mathbb{E}}

\renewcommand{\1}{\mathbbm{1}}
\newcommand{\Lb}{\mathcal{L}}

\newcommand{\esssup}{\mathrm{ess}\sup}

\newtheorem{Theorem}{Theorem}[section]
\newtheorem{Corollary}[Theorem]{Corollary}
\newtheorem{Lemma}[Theorem]{Lemma}
\newtheorem{Remark}[Theorem]{Remark}
\newtheorem{Definition}[Theorem]{Definition}
\newtheorem{Example}[Theorem]{Example}

\numberwithin{equation}{section}

\makeatletter
\newcommand{\customlabel}[1]{%
     \stepcounter{ref}%
   \protected@write
\@auxout{}{\string\newlabel{#1}{{\thesatz.\arabic{ref}}{\thepage}{\thesatz.\arabic{ref}}{#1}{}}}%
   \hypertarget{#1}{\thesatz.\arabic{ref}}%
}
\makeatother

\newenvironment{sciabstract}{\begin{quote}}{\end{quote}}

\newcounter{lastnote}

\title{Linear evolution equations in scales of Banach spaces}
\newcommand{\pdftitle} {Linear evolution equations in scales of Banach spaces}
\newcommand{\pdfauthor}{Martin Friesen}

\author{
Martin Friesen\footnote{Fakult\"at f\"ur Mathematik, Bergische Universit\"at Wuppertal, Germany, friesen@math.uni-wuppertal.de}\\
}

\usepackage[plainpages=false,pdfpagelabels=true,bookmarks=true,pdfauthor={\pdfauthor},
pdftitle={\pdftitle}]{hyperref}

\makeatletter
\def\HyPsd@CatcodeWarning#1{}
\makeatother

\begin{document}

\maketitle

\begin{sciabstract}\textbf{Abstract:}
This work is devoted to the study of a class of linear time-inhomogeneous evolution equations in a scale of Banach spaces.
Existence, uniquenss and stability for classical solutions is provided.
We study also the associated dual Cauchy problem for which we prove uniqueness in the dual scale of Banach spaces.
The results are applied to an infinite system of ordinary differential equations 
but also to the Fokker-Planck equation associated with the spatial logistic model in the continuum.
\end{sciabstract}


\section{Introduction}

\subsection{Motivation}
Interacting particle systems in the continuum can modeled on the phase space of locally finite configurations over $\R^d$, i.e.
\begin{align}\label{GAMMA}
 \Gamma = \{ \gamma \subset \R^d \ | \ |\gamma \cap \Lambda| < \infty \ \ \text{ for all compacts } \Lambda \subset \R^d \},
\end{align}
where $|A|$ denotes the number of elements in the set $A \subset \R^d$.
Markov dynamics on state space $\Gamma$ is usually described by a (heuristic) Markov operator $L$ 
acting on an appropriate class of functions $F: \Gamma \longrightarrow \R$. In the particular case of birth-and-death evolutions it has the general form
\begin{align}\label{BAD}
 (LF)(\gamma) = \sum \limits_{x \in \gamma}d(x,\gamma \backslash \{x\})(F(\gamma \backslash \{x\}) - F(\gamma))
 + \int \limits_{\R^d}b(x,\gamma)(F(\gamma \cup \{x\}) - F(\gamma))dx,
\end{align}
where $d \geq 0$ is the death-rate and $b \geq 0$ the birth-rate, respectively.
Other particular examples are discussed in \cite{FFHKKK15}, see also the references therein.
In some cases, the corresponding dynamics can be obtained from stochastic differential equations \cite{GK06}, 
other examples can be treated by the theory of Dirichlet forms \cite{AKR98, AKR98b, KL05}.
In most of the cases, however, one seeks to study dynamics in terms of state evolutions associated to the forward Kolmogorov equation
(also called Fokker-Planck equation)
\begin{align}\label{FPE}
 \frac{d}{dt}\int \limits_{\Gamma} F(\gamma)\mu_t(d\gamma) = \int \limits_{\Gamma}(LF)(\gamma) \mu_t(d\gamma), \ \ \mu_t|_{t = 0} = \mu_0,
\end{align}
where $F: \Gamma \longrightarrow \R$ belongs to a suitable space of test functions. Note that 
for $L$ of the form \eqref{BAD} one may take $F$ from
the class of polynomially bounded cylinder functions,
see Section 5 for details.
Since $\Gamma$ is an infinite dimensional, non-linear space the
study of \eqref{FPE} is a non-trivial mathematical task. 
Based on the work \cite{KKM08} the following appraoach 
for the study of the Fokker-Planck equation
was established in the literature.
First, we define for a reasonable class of states $\mu_t$ 
a sequence of correlation functions
$k_{\mu_t}^{(n)}: (\R^d)^n \longrightarrow \R_+$, $n \geq 0$,
by the relation
\[
 \int \limits_{\Gamma} \sum \limits_{\{ x_1, \dots, x_n\} \subset \gamma} f_n(x_1,\dots, x_n) \mu_t(d\gamma)
  = \frac{1}{n!}\int \limits_{(\R^d)^n} f_n(x_1,\dots, x_n) k_{\mu_t}^{(n)}(x_1,\dots, x_n) dx_1 \dots dx_n,
\]
where $f_n: (\R^d)^n \longrightarrow \R_+$ is symmetric, bounded and measurable. It can be shown that each
$k_{\mu_t}^{(n)}$ is symmetric and positive definite in the sense of Lenard. Moreover, there exists a one-to-one correspondence between such correlation functions and states, see, e.g., \cite{KK02}.
The study of \eqref{FPE} is therefore based on the study of the correlation function evolution $k_t := (k^{(n)}_{\mu_t})_{n=0}^{\infty}$ which should satisfy the Cauchy problem
\begin{align}\label{BBGKY}
 \frac{d k_{t}^{(n)}}{dt} = (L^{\Delta}k_{t})^{(n)}, \ \ k_{t}^{(n)}|_{t= 0} = k_{0}^{(n)}, \ \ n \geq 0.
\end{align}
Here $L^{\Delta}$ is an operator-valued matrix 
acting on the components of $k_t = (k_t^{(n)})_{n = 0}^{\infty}$.
Note that $L^{\Delta}$ can be computed explicitly
from $L$, see, e.g., \cite{FKO07} and \cite{FFHKKK15}.
Equation \eqref{BBGKY} can be seen as an Markov analogue of the BBGKY-hierarchy known from physics.
In contrast to \eqref{FPE}, equation \eqref{BBGKY} is 
an infinite system of finite-dimensional evolution equations
for which different functional analytic methods may be used.

Such an approach was successfully applied to various models \cite{KKP08, F11, FKK09, FKK13EST},
where the evolution of correlation functions was constructed in a weighted $L^{\infty}$-type space $\K_{\alpha}$ equipped with the norm
\begin{align*}
 \| k \|_{\mathcal{K}_{\alpha}} = \sup \limits_{n \geq 0} e^{- \alpha n} \| k^{(n)} \|_{L^{\infty}( (\R^d)^n)}, \ \ \alpha \in \R.
\end{align*}
A general semigroup approach is proposed in \cite{FKK12}, where, under some suitable conditions,
existence and uniqueness of classical solutions to \eqref{BBGKY} was shown. 
Since $\K_{\alpha}$ is a weighted $L^{\infty}$-Fock space, 
the main technical tool there is based on the study of the pre-dual Cauchy problem (evolution of quasi-observables)
and on a clever use of sun-dual spaces for strongly continuous semigroups. 
It is worth noting that existence of a solution to \eqref{BBGKY} does not immediately yield also a solution to \eqref{FPE}.
It is necessary and sufficient that $k_t$ is positive definite in the sense of Lennard, see \cite{KK02} and the references therein.
Such a positivity property was studied for particular models in \cite{FKK13, BKK15}, while a general result based on semigroup methods was obtained in \cite{FK16}. 
As a consequence, using the theory of semigroups one is able to obtain existence of solutions to \eqref{FPE},
while uniqueness holds among all solutions $\mu_t$ for which the associated sequence of correlation functions $k_{\mu_t}$ is a classical solution 
to \eqref{BBGKY}. It is worth noting that the weak formulation of the Fokker-Planck equation \eqref{FPE} does not require that $\mu_t$
also provides a classical solution to \eqref{BBGKY}. An extension of this uniqueness statement 
(without requiring that $k_{\mu_t}$ is a classical solution to \eqref{BBGKY}) was given in \cite{FK16, F17, FK18b}.
The latter result was essentially based on \cite{WZ02, WZ06}.

An application of semigroup methods often requires that the constant mortality rate $d(x,\{\emptyset\})$ is large enough, i.e. requires a
'high-mortality regime'. Going beyond this regime, one may sill construct an evolution of correlation functions 
obtained as the unique classical solution to \eqref{BBGKY} in the scale of Banach spaces $(\K_{\alpha})_{\alpha \in \R}$, 
see, e.g., \cite{BKKK13, FKKZ14, FKK15WRMODEL, KK18}.
Existence of a solution to \eqref{FPE} is again related with the possibility to show that this classical solution is positive definite in the sense of Lennard.
In such a case, one obtains a solution to \eqref{FPE} which is, in addition, unique among all solutions
for which the associated correlation function evolution provides a classical solution to \eqref{BBGKY} in the scale $(\K_{\alpha})_{\alpha \in \R}$.
An abstract formulation of the methods used in this approach was recently summarized in \cite{F15}. 

At present there does not exist any (general) uniqueness statement for \eqref{FPE} which does not require that the correlation function
evolution $k_{\mu_t}$ is a classical solution to \eqref{BBGKY}
in the scale $(\K_{\alpha})_{\alpha}$ as explained in Section 3 or in \cite{F15}.
The main purpose of this work is devoted to the study of \eqref{FPE} and \eqref{BBGKY} with main emphasis on
\begin{enumerate}
 \item[(i)] Study existence and uniqueness of classical solutions in the time-inhomogeneous case.
 \item[(ii)] Stability of the solutions with respect to the initial data and the parameters of the model.
 \item[(iii)] Uniqueness for solutions to \eqref{FPE} and \eqref{BBGKY} without assuming that $k_t$ is a classical solution.
\end{enumerate}
Points (i) and (ii) have been already studied in \cite{F15} for the time-homogeneous case.
Below we introduce an abstract formulation of this framework, and study the corresponding results in this abstract framework.
An application to the spatial logistic model which corresponds to a particular case of \eqref{FPE} is discussed in Section 5.

\subsection{Linear operators in a scale of Banach spaces}
Let $\E = (\E_{\alpha})_{\alpha > \alpha_*}$ be a scale of Banach spaces with $\alpha_* \in \R$ and
\begin{align}\label{LINEAR:29}
 \E_{\alpha} \subset \E_{\alpha'}, \ \ \Vert \cdot \Vert_{\alpha'} \leq \Vert \cdot \Vert_{\alpha}, \ \ \alpha' < \alpha.
\end{align}
Here and below we let $\alpha' < \alpha$ always stand for $\alpha > \alpha' > \alpha_*$.
Denote by $i_{\alpha \alpha'} \in L(\E_{\alpha}, \E_{\alpha'})$ the corresponding embedding operator. 
Here and in the following $L(\E_{\alpha}, \E_{\alpha'})$ stands for the space of all bounded linear operators from $\E_{\alpha}$ to $\E_{\alpha'}$,
and denote by $\Vert \cdot \Vert_{\alpha'\alpha}$ the corresponding operator norm on $L(\E_{\alpha}, \E_{\alpha'})$, $\alpha' < \alpha$.
Let $x=y$, $x \in \E_{\alpha'}$, $y \in \E_{\alpha}$ stand for $i_{\alpha \alpha'}x = y$.
\begin{Definition}
 A bounded linear operator $A$ in the scale $\E$ is, by definition, a collection of bounded linear operators from $\E_{\alpha}$ to $\E_{\alpha'}$, i.e.
 $A = (A_{\alpha \alpha'})_{\alpha' < \alpha}$, $A_{\alpha\alpha'} \in L(\E_{\alpha}, \E_{\alpha'})$, satisfying 
 \begin{align}\label{LINEAR:00}
  A_{\alpha'' \alpha'} = i_{\alpha \alpha'}A_{\alpha'' \alpha} = A_{\alpha \alpha'}i_{\alpha'' \alpha}, \ \ \alpha' < \alpha < \alpha''.
 \end{align}
\end{Definition}
By $A \in L(\E)$ we indicate that $L$ is a bounded linear operator in the scale $\E$.
Given two operators $A,B \in L(\E)$, the composition $AB \in L(\E)$ is, for all $\alpha' < \alpha$, defined by
\begin{align}\label{LINEAR:01}
 (AB)_{\alpha \alpha'} := A_{\beta \alpha'}B_{\alpha \beta},
\end{align}
where $\beta \in (\alpha', \alpha)$. 
It is worth noting that definition \eqref{LINEAR:01} does not depend on the particular choice of $\beta$, see \eqref{LINEAR:00}.
In view of \eqref{LINEAR:00} one finds that $(i_{\alpha \alpha'})_{\alpha ' < \alpha} \in L(\E)$ acts as an identity operator on $L(\E)$.
Note that such notion of linear operators includes the classical case of unbounded linear operators.
\begin{Example}
 For each $\alpha > \alpha_*$ let $(A_{\alpha}, D(A_{\alpha}))$ be a (possibly unbounded) linear operator in $\E_{\alpha}$. 
 Moreover, assume that $i_{\alpha \alpha'}(\E_{\alpha}) \subset D(A_{\alpha'})$ and
 \begin{align}\label{EQ:21}
  i_{\alpha'' \alpha'}A_{\alpha''}i_{\alpha \alpha''} = A_{\alpha'}i_{\alpha \alpha'} \in L(\E_{\alpha}, \E_{\alpha'}), \ \ \alpha' < \alpha'' < \alpha.
 \end{align}
 Then $A_{\alpha \alpha'} := A_{\alpha'}i_{\alpha \alpha'}$ defines an element $A = (A_{\alpha \alpha'})_{\alpha' < \alpha} \in L(\E)$.
\end{Example}

\subsection{Linear evolution equations in a scale of Banach spaces}
Consider a family of operators $(A(t))_{t \geq 0} \subset L(\E)$ with the property
\begin{enumerate}
 \item[(A1)] For all $\alpha' < \alpha$, $t \longmapsto A_{\alpha\alpha'}(t) \in L(\E_{\alpha}, \E_{\alpha'})$ is continuous in the operator norm.
\end{enumerate}
We suppose that it is associated to a forward (or backward) evolution system on the scale $\E$.
The precise conditions are formulated and discussed in Section 2.
Let us just mention that such conditions are more general then classical evolution systems studied as studied, e.g., in \cite{PAZ83}. They are designed, in particular,
to apply to scales of weighted $L^{\infty}$-spaces, see also Section 5.
Let $(B(t))_{t \geq 0} \subset L(\E)$ be another family of operators with the properties
\begin{enumerate}
 \item[(B1)] For all $\alpha' < \alpha$ and $k \in \E_{\alpha}$, $t \longmapsto B_{\alpha \alpha'}(t)k \in \E_{\alpha'}$ is continuous.
 \item[(B2)] There exists an increasing continuous function $M: (\alpha_*, \infty) \longrightarrow [0,\infty)$ with
 \begin{align}\label{LINEAR:33}
  \Vert B_{\alpha \alpha'}(t)\Vert_{\alpha \alpha'} \leq \frac{M(\alpha)}{\alpha - \alpha'}, \ \ t \geq 0, \ \ \alpha' < \alpha.
 \end{align}
\end{enumerate}
We are interested in solutions to the forward and backward Cauchy problems
\begin{align*}
 \frac{\partial u(t)}{\partial t} &= (A(t) + B(t))u(t), \ \ u(s) = k, \ \ t \geq s,
 \\ \frac{\partial v(s)}{\partial s} &= - (A(s) + B(s))v(s), \ \ v(t) = k, \ \ s \in [0,t]
\end{align*}
in the scale $\E$. The precise definitions and results are formulated in Section 3.
It is worthwhile to mention that similar equations have been recently studied in the case where $B(t)$ was a nonlinear operator \cite{SAF95}, \cite{FK18}.
Such equations with operators $A(t), B(t)$ as above have also applications to partial differential equations, 
see \cite{T08, CAPS, H13, BHP15}.
We close this presentation with one simple example
for which the results obtained in the subsequent sections can be applied. 
\begin{Example}
 Let $(X, \mathcal{B}(X), \mu)$ be a $\sigma$-finite Borel space
 and take a family of functions
 $\omega_{\alpha}:X \longrightarrow [0,\infty)$ indexed by $\alpha \in \R$ satisfying
 \[
  \omega_{\alpha'} \leq \omega_{\alpha} \ \ \text{ and }\ \ \esssup \limits_{x \in X}m(x) \frac{\omega_{\alpha'}(x)}{\omega_{\alpha}(x)} < \infty, \qquad \forall \alpha' < \alpha.
 \]
 Define $\E_{\alpha} = L^1(X, \omega_{\alpha}\mu)$ with 
 norm $\| f\|_{\alpha} = \int_X |f(x)| \omega_{\alpha}(x)\mu(dx)$.
 Consider the objects
 \begin{enumerate}
   \item[(i)] $m: X \longrightarrow [0,\infty)$ is
   measurable and locally bounded.
   \item[(ii)] $k: X \times X \longrightarrow \R$ is measurable
   and, for all $\alpha' < \alpha$, there exists an increasing function $M(\alpha) > 0$ with
   \[
    \int \limits_{X}|k(x,y)| \omega_{\alpha'}(x)\mu(dx) \leq \frac{M(\alpha)}{\alpha - \alpha'}\omega_{\alpha}(y), \qquad \mu-\text{ a.a. } y \in X.
   \]
 \end{enumerate} 
 Then $A$ and $B$ defined by
 \[
  (Af)(x) := -m(x)f(x)\ \ \text{ and } \ \ (Bf)(x) = \int \limits_{X}f(y)k(x,y)\mu(dy) 
 \]
 satisfy properties (A1), (B1) and (B2).
 Moreover, the semigroup generated by $A$ satisfies the conditions formulated in Section 2.
\end{Example}
More delicate 
examples and applications are discussed in Sections 4 and 5.

\subsection{Structure of the work}
This work is organized as follows. 
In Section 2 we study some basic properties of forward and backward evolution systems associated to $(A(t))_{t \geq 0}$ given by (A1).
The main results of this work, that is existence, uniqueness, stability and the dual Cauchy problems, are formulated and proved in Section 3.
Applications to a system of ordinary differential equations and Markov evolutions in the continuum are considered in Sections 4 and 5, respectively.

\section{Some simple properties for evolution systems}
Let $\E = (\E_{\alpha})_{\alpha > \alpha_*}$ be a scale of Banach spaces in the sense of \eqref{LINEAR:29}.
Below we first introduce the notion of a forward evolution system associated to $(A(t))_{t \geq 0}$ and give some simple properties used Section 3.
Afterwards we state the corresponding results (without proofs) for the backward evolution systems.

\subsection{The forward evolution system}
Let $(A(t))_{t \geq 0}$ be given as in (A1).
Consider a family of operators $(V(t,s))_{t \geq s \geq 0}$ on the scale $\E$ satisfying the following conditions
\begin{enumerate}
 \item[(A2)] We have $(V(t,s))_{t \geq s \geq 0} \subset L(\E)$, and $V_{\alpha \alpha'}(t,t) = i_{\alpha \alpha'}$, for all $\alpha' < \alpha$ and $t \geq 0$.
 Moreover, for any $\alpha' <  \alpha$ and $k \in \E_{\alpha}$, the mapping $(t,s) \longmapsto V_{\alpha \alpha'}(t,s)k$ is continuous in $\E_{\alpha'}$.
 \item[(A3)] For any $\alpha' < \alpha'' < \alpha$, $k \in \E_{\alpha}$ and $0 \leq s \leq t$, 
  \begin{align*}
   V_{\alpha \alpha'}(t,s)k = k + \int \limits_{s}^{t}A_{\alpha'' \alpha'}(r)V_{\alpha \alpha''}(r,s)k dr, \qquad
   V_{\alpha \alpha'}(t,s)k = k - \int \limits_{s}^{t}V_{\alpha'' \alpha'}(t,r)A_{\alpha \alpha''}(r)k dr,
  \end{align*}
 where the integrals exist in $\E_{\alpha'}$.
\end{enumerate}
The following properties are immediate consequences of (A1) -- (A3).
\begin{Lemma}\label{LINEARLEMMA:00}
 Let $(A(t))_{t \geq 0}$ be given as in (A1) and let $(V(t,s))_{t \geq s \geq 0}$ satisfy (A2) and (A3). 
 Then the following assertions hold:
 \begin{enumerate}
  \item[(a)] For any $\alpha' < \alpha$, $k \in \E_{\alpha}$ and $s \geq 0$, the mapping $[s,\infty) \ni t \longmapsto V_{\alpha \alpha'}(t,s)k$
  is continuously differentiable in $\E_{\alpha'}$, and we have, for all $\alpha'' \in (\alpha', \alpha)$,
 \begin{align}
  \label{LINEAR:30} \frac{\partial }{\partial t}V_{\alpha\alpha'}(t,s)k &= A_{\alpha'' \alpha'}(t)V_{\alpha \alpha''}(t,s)k.
 \end{align}
  \item[(b)] For any $\alpha' < \alpha$, $k \in \E_{\alpha}$ and $t > 0$, the mapping $[0,t] \ni s \longmapsto V_{\alpha \alpha'}(t,s)k$
  is continuously differentiable in $\E_{\alpha'}$, and we have for all $\alpha'' \in (\alpha', \alpha)$
 \begin{align}
   \label{LINEAR:31} \frac{\partial }{\partial s}V_{\alpha \alpha'}(t,s)k &= - V_{\alpha'' \alpha'}(t,s)A_{\alpha \alpha''}(s)k.
 \end{align}
  \item[(c)] $(V(t,s))_{t \geq s \geq 0}$ is uniquely determined by \eqref{LINEAR:30} and \eqref{LINEAR:31}.
  \item[(d)] Suppose that, for all $\alpha' < \alpha$, the mapping $t \longmapsto A_{\alpha\alpha'}(t)$ is continuous with respect to the operator norm on  $L(\E_{\alpha}, \E_{\alpha'})$. Then assertions (a) and (b) also hold with respect to the operator norm on $L(\E_{\alpha}, \E_{\alpha'})$ for all $\alpha'  < \alpha$.
 \end{enumerate}
\end{Lemma}
 \begin{proof}
 Using (A1) and (A2) one obtains that the integrals in (A3) are continuous in $r$. Hence, the left-hand sides in (A3) are continuously
 differentiable which yields assertions \textit{(a)} and \textit{(b)}.
 Concerning \textit{(c)}, let $(\widetilde{V}(t,s))_{t \geq s \geq 0}$ also satisfy (A2) and (A3) with the same $(A(t))_{t \geq 0}$.
 Take $\alpha' < \alpha_0 < \alpha_1 < \alpha$, $k \in \E_{\alpha}$ and $0 \leq s < t$, 
 then $[s,t] \ni r \longmapsto V_{\alpha_0\alpha'}(t,r)i_{\alpha_1 \alpha_0}\widetilde{V}_{\alpha\alpha_1}(r,s)k$ is continuously differentiable in $\E_{\alpha'}$ 
 such that $\frac{d}{dr}V_{\alpha_0\alpha'}(t,r)i_{\alpha_1 \alpha_0}\widetilde{V}_{\alpha\alpha_1}(r,s)k = 0$.
 Integrating over $r \in [s,t]$ gives 
 \[
 0 = V_{\alpha'' \alpha'}(t,t)\widetilde{V}_{\alpha\alpha''}(t,s)k - V_{\alpha''\alpha}(t,s)\widetilde{V}_{\alpha \alpha''}(s,s)k 
   = \widetilde{V}_{\alpha \alpha'}(t,s)k - V_{\alpha \alpha'}(t,s)k,
 \]
 where we have used $V(t,s), \widetilde{V}(t,s) \in L(\E)$ and property (A2). This proves \textit{(c)}.
 Let us prove \textit{(d)}. Using (A3) we obtain, for all $\alpha' < \alpha'' < \alpha$, $k \in \E_{\alpha}$ and $h \in [0,1]$ small enough, 
 \[
  \| V_{\alpha \alpha'}(t+h,s)k - V_{\alpha \alpha'}(t,s)k\|_{\alpha'} \leq h \sup \limits_{r \in [t,t+1]} \| A_{\alpha'' \alpha}(r) \| _{\alpha'' \alpha'} \sup \limits_{r \in [t,t+1]}\| V_{\alpha \alpha''}(r,s) \|_{\alpha \alpha''} \| k \|_{\alpha},
 \]
 and a similar estimate for $h \in [-1,0]$ close enough to $0$.
 Hence $t \longmapsto V_{\alpha \alpha'}(t,s) \in L(\E_{\alpha}, \E_{\alpha'})$ is continuous in the operator norm.
 Similarly one can show that $s \longmapsto V_{\alpha \alpha'}(t,s)$ is continuous in the operator norm.
 From this and the continuity of $t \longmapsto A_{\alpha\alpha'}(t)$ we deduce that the integrals in (A3) exist in the operator norm on $L(\E_{\alpha}, \E_{\alpha'})$.
 This proves the assertion.
\end{proof}
The following remark justifies the name forward evolution system for $(V(t,s))_{t \geq s \geq 0}$.
\begin{Remark}
 Suppose that (A1) -- (A3) are satisfied.
 Then, for all $t \geq r \geq s \geq 0$, it holds that $V(t,s) = V(t,r)V(r,s)$, where the composition is defined by \eqref{LINEAR:00}, i.e.
 \[
  V_{\alpha \alpha'}(t,s) = V_{\alpha'' \alpha'}(t,r)V_{\alpha \alpha''}(r,s), \ \ \alpha' < \alpha'' < \alpha.
 \]
 This property can be directly deduced from previous Lemma. 
 Since it is also a particular case of the results discussed in Section 3, we omit here the proof.
\end{Remark}
Next we provide a simple stability estimate for the forward evolution system $(V(t,s))_{t \geq s \geq 0}$.
\begin{Lemma}\label{LEMMA:STABILITY}
 Let $(V(t,s))_{t \geq s \geq 0}$, $(A(t))_{t \geq 0}$ and $(\widetilde{V}(t,s))_{t \geq s \geq 0}$, $(\widetilde{A}(t))_{t \geq 0}$
 both satisfy (A1) -- (A3). Moreover, suppose that, for all $\alpha' < \alpha$, the mappings
 $t \longmapsto A_{\alpha \alpha'}(t)$ and $t \longmapsto \widetilde{A}_{\alpha \alpha'}(t)$
 are continuous with respect to the operator norm on $L(\E_{\alpha}, \E_{\alpha'})$.
 Then, for all $t \geq s \geq 0$ and $\alpha ' < \alpha_0 < \alpha_1 < \alpha$, one has
 \[
  \| V_{\alpha \alpha'}(t,s) - \widetilde{V}_{\alpha \alpha'}(t,s) \|_{\alpha \alpha'} 
  \leq  C(\alpha', \alpha_0, \alpha_1, \alpha,t,s) \int \limits_{s}^{t} \| A_{\alpha_1 \alpha_0}(r) - \widetilde{A}_{\alpha_1\alpha_0}(r) \|_{\alpha_1 \alpha_0}dr,
 \]
 where $C(\alpha', \alpha_0, \alpha_1, \alpha,t,s) = \sup_{r \in [s,t]} \| V_{\alpha_0 \alpha'}(t,r) \|_{\alpha_0 \alpha'} \cdot \sup_{r \in [s,t]} \| \widetilde{V}_{\alpha \alpha_1}(r,s) \|_{\alpha \alpha_1}$ is finite due to (A2) and the uniform boundedness principle.
\end{Lemma}
\begin{proof}
 Using Lemma \ref{LINEARLEMMA:00}.(d) we see that 
 $[s,t] \ni r \longmapsto V_{\alpha_0 \alpha'}(t,r)  i_{\alpha_1 \alpha_0} \widetilde{V}_{\alpha \alpha_1}(r,s)$ is continuously differentiable in $L(\E_{\alpha}, \E_{\alpha'})$
 and satisfies
 \begin{align*}
   \frac{d}{dr}V_{\alpha_0 \alpha'}(t,r)  i_{\alpha_1 \alpha_0} \widetilde{V}_{\alpha \alpha_1}(r,s) 
   &= V_{\alpha_0 \alpha'}(t,r)\left( - A_{\alpha_1 \alpha_0}(r) + \widetilde{A}_{\alpha_1 \alpha_0}(r) \right) \widetilde{V}_{\alpha \alpha_1}(r,s).
 \end{align*}
 Integrating over $r \in [s,t]$, using $V(t,s), \widetilde{V}(t,s) \in L(\E)$ and (A2) gives
 \begin{align*}
  \widetilde{V}_{\alpha \alpha'}(t,s) - V_{\alpha \alpha'}(t,s) = \int \limits_{s}^{t} V_{\alpha_0 \alpha'}(t,r)\left( - A_{\alpha_1 \alpha_0}(r) + \widetilde{A}_{\alpha_1 \alpha_0}(r) \right) \widetilde{V}_{\alpha \alpha_1}(r,s)dr.
 \end{align*}
 Taking the norm $\| \cdot \|_{\alpha \alpha'}$ and using the triangle inequality for the integral proves the assertion.
\end{proof}
The next remark gives a sufficient condition for classical forward evolution systems to satisfy conditions (A1) -- (A3), see \cite{PAZ83, KOL13} for their study.
\begin{Remark}
 For each $\alpha > \alpha_*$ let $(V_{\alpha}(t,s))_{t \geq s \geq 0} \subset L(\E_{\alpha})$, and let $(A_{\alpha}(t), D(A_{\alpha}(t))$ be linear operators on $\E_{\alpha}$ satisfying the following properties
 \begin{enumerate}
  \item[(i)] $V_{\alpha}(t,t) = 1$, $V_{\alpha}(t,r)V_{\alpha}(r,s) = V_{\alpha}(t,s)$ for all $t \geq r \geq s \geq 0$ and $\alpha > \alpha_*$.
  \item[(ii)] $V_{\alpha'}(t,s)|_{\E_{\alpha}} = i_{\alpha \alpha'}V_{\alpha}(t,s)$ for any $\alpha' < \alpha$ and $t \geq s \geq 0$.
  \item[(iii)] $(t,s) \longmapsto V_{\alpha'}(t,s)k \in \E_{\alpha'}$ is continuous for any $\alpha' < \alpha$ and $k \in \E_{\alpha}$.
  \item[(iv)] $i_{\alpha \alpha'}(\E_{\alpha}) \subset D(A_{\alpha'}(t))$ and \eqref{EQ:21} hold for all $t \geq 0$.
  Moreover, $[0,\infty) \ni t \longmapsto A_{\alpha'}(t)i_{\alpha \alpha'} \in L(\E_{\alpha}, \E_{\alpha'})$ is strongly continuous for all $\alpha' < \alpha$.
 \end{enumerate}
 Then (A1) and (A2) are satisfied by 
 \[
  V_{\alpha \alpha'}(t,s) = V_{\alpha'}(t,s) i_{\alpha \alpha'}, \qquad A_{\alpha \alpha'}(t) = A_{\alpha'}(t)i_{\alpha \alpha'}, \ \ \alpha' < \alpha, \ \ t \geq 0.
 \]
 Suppose, in addition, the condition
 \begin{enumerate}
  \item[(v)] For all $\alpha' < \alpha$, $t \geq s \geq 0$ and $k \in \E_{\alpha}$,
  \[
   V_{\alpha'}(t,s)k = k + \int \limits_{s}^{t}A_{\alpha'}V_{\alpha'}(r,s)i_{\alpha \alpha'} k dr, \qquad
   V_{\alpha'}(t,s)k = k - \int \limits_{s}^{t}V_{\alpha'}(t,r)A_{\alpha'}(r)i_{\alpha \alpha'}k dr.
  \]
 \end{enumerate}
 then also property (A3) is satisfied.
\end{Remark}
In the time-homogeneous case, i.e. $A(t)$ is independent of $t$, 
above conditions can be simplyfied as follows.
\begin{Remark}\label{TIMEHOMOGENEOUS}
 For each $\alpha > \alpha_*$ let $(T_{\alpha}(t))_{t \geq 0} \subset L(\E_{\alpha})$ be a strongly continuous semigroup with generator $(A_{\alpha},D(A_{\alpha}))$ on $\E_{\alpha}$. 
 Suppose that
 \begin{enumerate}
  \item[(i)] $T_{\alpha'}(t)|_{\E_{\alpha}} = T_{\alpha}(t)$, for all $t \geq 0$ and $\alpha' < \alpha$.
  \item[(ii)] $\E_{\alpha} \subset D(A_{\alpha'})$ and $A_{\alpha'}i_{\alpha \alpha'} \in L(\E_{\alpha}, \E_{\alpha'})$.
 \end{enumerate}
 Combining properties (i) and (ii) together with \cite[p.123, Theorem 5.5]{PAZ83} gives $A_{\alpha}k = A_{\alpha'}k$ for all
 \[
  \ k \in D(A_{\alpha}) = \{ h \in D(A_{\alpha'}) \cap \E_{\alpha} \ | \ A_{\alpha'}h \in \E_{\alpha} \}
  = \{ h \in \E_{\alpha} \| \ A_{\alpha'}h \in \E_{\alpha} \}.
 \]
 Hence $A_{\alpha \alpha'} := A_{\alpha'}i_{\alpha \alpha'}$ and $V_{\alpha \alpha'}(t,s) := T_{\alpha\alpha'}(t-s) := T_{\alpha'}(t-s)i_{\alpha \alpha'}$
 satisfy properties (A1) -- (A3). Moreover, one can show that, for any $\alpha' < \alpha$ and $k \in \E_{\alpha}$, the function
$t \longmapsto T_{\alpha\alpha'}(t)k$ is infinitely often differentiable in $\E_{\alpha'}$ such that, for any $n \geq 1$,
\[
 \frac{d^n}{dt^n}T_{\alpha\alpha'}(t)k = A^n_{\alpha'' \alpha'}T_{\alpha \alpha''}(t)k = T_{\alpha''\alpha'}(t)A^n_{\alpha \alpha''}k, \qquad t \geq 0.
\]
\end{Remark}

\subsection{The backward evolution system}
Consider a family of operators $(V(s,t))_{t \geq s \geq 0}$ on the scale $\E$ satisfying the following conditions
\begin{enumerate}
 \item[(A2)$^*$] We have $(V(s,t))_{t \geq s \geq 0} \subset L(\E)$ with $V_{\alpha \alpha'}(t,t) = i_{\alpha \alpha'}$ for all $\alpha' < \alpha$ and $t \geq 0$.
  Moreover, for any $\alpha' <  \alpha$ and $k \in \E_{\alpha}$, the mapping $(s,t) \longmapsto V_{\alpha \alpha'}(s,t)k$ is continuous in $\E_{\alpha'}$.
 \item[(A3)$^*$] For any $\alpha' < \alpha'' < \alpha$, $k \in \E_{\alpha}$ and $0 \leq s \leq t$, 
  \begin{align*}
   V_{\alpha \alpha'}(s,t)k = k + \int \limits_{s}^{t}V_{\alpha'' \alpha'}(s,r)A_{\alpha \alpha''}(r)k dr, \qquad
   V_{\alpha \alpha'}(s,t)k = k - \int \limits_{s}^{t}A_{\alpha'' \alpha'}(r)V_{\alpha \alpha''}(r,t)k dr.
  \end{align*}
\end{enumerate}
The following properties are immediate consequences of (A1), and (A2)$^*$, (A3)$^*$
and can be deduced by similar arguments to Lemma \ref{LINEARLEMMA:00}.
\begin{Lemma}\label{LINEARLEMMA:10}
 Let $(A(t))_{t \geq 0}$ be given by (A1) and $(V(s,t))_{t \geq s \geq 0}$ satisfy properties (A2)$^*$,(A3)$^*$.
 Then the following assertions hold
 \begin{enumerate}
  \item[(a)] For any $\alpha' < \alpha$, $k \in \E_{\alpha}$ and $s \geq 0$, the mapping $[s,\infty) \ni t \longmapsto V_{\alpha \alpha'}(s,t)k$
  is continuously differentiable in $\E_{\alpha'}$, such that for all $\alpha'' \in (\alpha', \alpha)$,
 \begin{align}\label{EQ:09} 
  \frac{\partial }{\partial t}V_{\alpha\alpha'}(t,s)k &= V_{\alpha'' \alpha'}(s,t)A_{\alpha \alpha''}(t)k.
 \end{align}
  \item[(b)] For any $\alpha' < \alpha$, $k \in \E_{\alpha}$ and $t > 0$, the mapping $[0,t] \ni s \longmapsto V_{\alpha \alpha'}(s,t)k$
  is continuously differentiable in $\E_{\alpha'}$, such that for all $\alpha'' \in (\alpha', \alpha)$,
 \begin{align}\label{EQ:10} 
   \frac{\partial }{\partial s}V_{\alpha \alpha'}(t,s)k &= - A_{\alpha'' \alpha'}(s)V_{\alpha \alpha''}(s,t)k.
 \end{align}
  \item[(c)] $(V(s,t))_{t \geq s \geq 0}$ is uniquely determined by \eqref{EQ:09} and \eqref{EQ:10}.
  \item[(d)] Suppose that, for all $\alpha' < \alpha$, the mapping $t \longmapsto A_{\alpha\alpha'}(t)$ is continuous with respect to the operator norm on  $L(\E_{\alpha}, \E_{\alpha'})$. Then assertions (a) and (b) also hold with respect to the operator norm on $L(\E_{\alpha}, \E_{\alpha'})$ for all $\alpha'  < \alpha$.
 \end{enumerate}
\end{Lemma}
As before, the next remark explains why $(V(s,t))_{t \geq s \geq 0}$ is called backward evolution system.
\begin{Remark}
 Suppose that (A1), (A2)$^*$ and (A3)$^*$ are satisfied. Then $V(s,t) = V(s,r)V(r,t)$ holds for all $t \geq r \geq s \geq 0$ in the sense of \eqref{LINEAR:00}.
\end{Remark}
Clearly a similar estimate to the one proved in Lemma \ref{LEMMA:STABILITY} can be also obtained in this case.

\section{Construction of the perturbation}
\subsection{Construction of forward evolution system}
Let $(V(t,s))_{t  \geq s \geq 0}$ be a family of operators as in (A2). Suppose, in addition, that 
\begin{enumerate}
  \item[(A4)] There exists a constant $K \geq 1$ such that for all $\alpha' < \alpha$
 \begin{align*}
  \Vert V_{\alpha \alpha'}(t,s)\Vert_{\alpha \alpha'} \leq K, \ \ 0 \leq s \leq t.
 \end{align*}
\end{enumerate}
Below we provide the construction of a perturbed family of operators 
\begin{align}\label{EQ:15}
  \{ W_{\alpha \alpha'}(t,s) \in L(\E_{\alpha}, \E_{\alpha'}) \ | \ \alpha' < \alpha, \ \ 0 \leq t - s < T(\alpha', \alpha) \},
\end{align}
where $T(\alpha', \alpha) := \frac{\alpha - \alpha'}{2Ke M(\alpha)}$, $\alpha' < \alpha$, with the convention that $1 / 0 := + \infty$.
Such a construction is based on the Ovsyannikov technique, see e.g. \cite{F15} and \cite{FK18} for some recent related results.
\begin{Theorem}\label{THEOREM:00}
 Suppose that $(V(t,s))_{t \geq s \geq 0}$ satisfies (A2), (A4) and $(B(t))_{t \geq 0}$ satisfies (B1), (B2).
 Then there exists a family of operators $W(t,s)$ as in \eqref{EQ:15} with the properties
 \begin{enumerate}
  \item[(a)] For any $\alpha' < \alpha$ and $0 \leq t-s < T(\alpha', \alpha)$,
  $(t,s) \longmapsto W_{\alpha \alpha'}(t,s)$ is strongly continuous on $L(\E_{\alpha}, \E_{\alpha'})$ with $W_{\alpha \alpha'}(t,t) = i_{\alpha \alpha'}$, and
  \[
   \Vert W_{\alpha \alpha'}(t,s) \Vert_{\alpha \alpha'} \leq  \frac{T(\alpha', \alpha)}{T(\alpha',\alpha) - (t-s)}K.
  \]
  \item[(b)] For all $\alpha' < \alpha'' < \alpha$ with $0 \leq t - s < \min\{ T(\alpha', \alpha), T(\alpha'', \alpha)\}$ we have
  \[
   W_{\alpha \alpha'}(t,s) = i_{\alpha'' \alpha'}W_{\alpha \alpha''}(t,s),
  \]
  and for  all $\alpha' < \alpha'' < \alpha$ with $0 \leq t - s < \min\{ T(\alpha', \alpha), T(\alpha', \alpha'')\}$ we have
  \[
   W_{\alpha \alpha'}(t,s) = W_{\alpha'' \alpha'}(t,s)i_{\alpha \alpha''}.
  \]
  \item[(c)] For all $\alpha' < \alpha'' < \alpha$ with $0 \leq t - s < \min\{ T(\alpha', \alpha), T(\alpha'', \alpha)\}$ we have
  \begin{align}\label{EQ:06}
   W_{\alpha \alpha'}(t,s)k = V_{\alpha \alpha'}(t,s)k + \int \limits_{s}^{t}(V(t,r)B(r))_{\alpha'' \alpha'}W_{\alpha \alpha''}(r,s)k dr, \ \ \forall k \in \E_{\alpha}.
  \end{align}
   Moreover $W(t,s)$ is unique with such property.
  \item[(d)] For all $\alpha' < \alpha'' < \alpha$ with $0 \leq t - s < \min\{ T(\alpha', \alpha), T(\alpha', \alpha'')\}$ we have
  \begin{align*}
   W_{\alpha \alpha'}(t,s)k = V_{\alpha \alpha'}(t,s)k + \int \limits_{s}^{t}W_{\alpha'' \alpha'}(t,r)(B(r)V(r,s))_{\alpha \alpha''}k dr, \ \ \forall k \in \E_{\alpha}.
  \end{align*}
   Moreover $W(t,s)$ is also unique with such property.
 \end{enumerate}
\end{Theorem}
Recall \eqref{LINEAR:00} and \eqref{LINEAR:01}.
In order to simply the notation in the proof, we omit the subscripts $\alpha \alpha'$, whenever no confusion may arise.
\begin{proof}
 Define a sequence of operators $(W_n(t,s))_{0 \leq s \leq t} \subset L(\E)$ by $W_0(t,s) = V(t,s)$ and 
 \begin{align}\label{LINEAR:18}
  W_{n+1}(t,s) := \int \limits_{s}^{t}V(t,r)B(r)W_n(r,s) dr,
 \end{align}
 where the integrals are, for all $\alpha' < \alpha$, defined in the strong topology on $L(\E_{\alpha}, \E_{\alpha'})$,
 while the composition $V(t,r)B(r)W_n(r,s)$ is defined by \eqref{LINEAR:01}. 
 Then, for any $\alpha' < \alpha$, $n \geq 0$ and $k \in \E_{\alpha}$, the function $W_n(t,s)k$ is continuous in $\E_{\alpha'}$ and satisfies
 \[
  \Vert W_n(t,s)k\Vert_{\alpha'} \leq \Vert k \Vert_{\alpha} \left(\frac{t-s}{T(\alpha',\alpha)}\right)^nK.
 \]
 Indeed, for $n = 0$ this certainly holds true due to (A4).
 Consider $n \geq 1$, set $\alpha_j := \alpha' + j \frac{\alpha - \alpha'}{2n}$, $j \in \{0, \dots, 2n \}$ and for $s \leq t_1 \leq \dots \leq t_{n} \leq t$ let
 \[
  Q_n(t,t_1,\dots, t_n,s) := V(t,t_1)B(t_1)\cdots V(t_{2n-2}, t_{2n-1} ) B(t_{2n-1}) V(t_{2n},s) \in L(\E_{\alpha}, \E_{\alpha'}),
 \]
 where the composition is again defined by \eqref{LINEAR:01}.
 Using \eqref{LINEAR:33} and (A4), we obtain by \eqref{LINEAR:18}
 \begin{align*}
  \Vert W_n(t,s)k\Vert_{\alpha'} &\leq \int \limits_{s}^{t}\cdots \int \limits_{s}^{t_{n-1}}\Vert Q_n(t,t_1,\dots, t_n,s)k \Vert_{\alpha'}dt_n \dots dt_1
  \\ &\leq K^{n} \Vert k \Vert_{\alpha}\frac{(2n)^n}{(\alpha - \alpha')^n} \int \limits_{s}^{t}\cdots \int \limits_{s}^{t_{n-1}}\prod \limits_{j=0}^{n-1}M(\alpha_{2j+1}) dt_n \dots dt_1
  \\ &\leq \Vert k \Vert_{\alpha} \frac{(t-s)^n}{n!} \frac{(2M(\alpha)n K)^n}{(\alpha-\alpha')^n}
  \\ &\leq \| k \|_{\alpha} \left( \frac{t-s}{T(\alpha', \alpha)} \right)^n,
 \end{align*}
 where we have used $\frac{1}{n!} \leq \left( \frac{e}{n}\right)^n$, $n \geq 1$. 
 Hence $\sum_{n=0}^{\infty}W_n(t,s)k =: W(t,s)k$ converges locally uniformly in $\E_{\alpha'}$
 for all $0 \leq t - s  < T(\alpha', \alpha)$, i.e. $(t,s) \longmapsto W(t,s)k \in \E_{\alpha'}$ is continuous and satisfies
 \begin{align*}
  \Vert W(t,s)k \Vert_{\alpha'} &\leq \Vert k \Vert_{\alpha} \sum \limits_{n=0}^{\infty} \left( \frac{t-s}{T(\alpha', \alpha)}\right)^n K
  = \Vert k \Vert_{\alpha}\frac{T(\alpha', \alpha)}{T(\alpha',\alpha) - (t-s)}K.
 \end{align*}
 This proves property (a). Property (b) is a direct consequence of 
 \[
   (W_n(t,s))_{\alpha' \alpha} = i_{\alpha'' \alpha'}(W_n(t,s))_{\alpha \alpha''} = (W_n(t,s))_{\alpha'' \alpha'}i_{\alpha \alpha''}.
 \]
 Let us prove property (c). 
 Take $\alpha_j := \alpha' + j\frac{\alpha - \alpha'}{2(n+1)}$, $j \in \{0, \dots, 2(n+1)\}$. Then, for $s \leq r \leq t$, we obtain
 \begin{align*}
  \Vert V(t,r)B(r)W_n(r,s)k \Vert_{\alpha'} 
  &\leq  (KM(\alpha))^{n+1} \frac{2(n+1)}{\alpha - \alpha'} \frac{(t-s)^n}{n!} \frac{(2(n+1))^n}{(\alpha - \alpha')^n}\Vert k \Vert_{\alpha}
  \\ &\leq  \Vert k \Vert_{\alpha} \frac{4e K M(\alpha)}{\alpha - \alpha'} \left( \frac{t-s}{T(\alpha', \alpha)}\right)^{n} n.
 \end{align*}
 Hence the series $\sum_{n = 0}^{\infty}V(t,r) B(r)W_n(r,s)k \in \E_{\alpha'}$ is locally uniformly convergent in $s \leq r \leq t$,
 provided one has $0 \leq t - s < T(\alpha', \alpha)$. Take $\alpha' < \alpha'' < \alpha$ with $0 \leq t-s < \min\{ T(\alpha'', \alpha'), T(\alpha', \alpha)\}$.
 Note that for given $t - s < T(\alpha', \alpha)$ such $\alpha''$ always exists.
 Then $W(r,s)k = \sum_{n =0}^{\infty}W_n(r,s)k$ converges locally uniformly in $\E_{\alpha''}$ and hence is also continuous in $r$.
 Since $V(t,r)B(r) \in L(\E_{\alpha''}, \E_{\alpha'})$ is strongly continuous, it follows that
 \begin{align*}
  W_{\alpha \alpha'}(t,s)k &= V_{\alpha \alpha'}(t,s)k + \sum \limits_{n=1}^{\infty}\int \limits_{s}^{t}(V(t,r)B(r))_{\alpha'' \alpha'}(W_{n-1}(r,s))_{\alpha \alpha''} kdr
  \\ &= V_{\alpha \alpha'}(t,s)k + \int \limits_{s}^{t}(V(t,r)B(r))_{\alpha'' \alpha'}W_{\alpha \alpha''}(r,s)k dr,
 \end{align*}
 which implies that $W(t,s)$ satisfies the desired integral equation \eqref{EQ:06}. Let us prove that $W(t,s)$ is unique with such property.
 Take another family of operators $\widetilde{W}(t,s)$ satisfying (a), (b) and \eqref{EQ:06}.
 Let $\alpha' < \alpha'' < \alpha$ with $0 \leq t-s < \min\{ T(\alpha'', \alpha), T(\alpha', \alpha)\}$
 and set $\alpha_j = \alpha' + j \frac{\alpha'' - \alpha'}{2n}$, $j \in \{0,\dots, 2n\}$ where $n \geq 1$.
 Observe that, for all $s \leq t_n \leq t_{n-1} \leq \dots \leq t_1 \leq t$, one has
 \begin{align*}
  \widetilde{Q}_n(t,t_1, \dots, t_n) := V(t,t_1)B(t_1)\cdots V(t_{n-1}, t_n)B(t_n) \in L(\E_{\alpha''}, \E_{\alpha'}).
 \end{align*}
 Then, as before, we obtain
 \begin{align*}
  \Vert \widetilde{Q}_n(t,t_1, \dots, t_n)\Vert_{\alpha' \alpha''} \leq K^{n} \frac{M(\alpha'')^n(2n)^n}{(\alpha'' - \alpha')^n}
 \end{align*}
 and hence we deduce that, for all $k \in \E_{\alpha}$,
 \begin{align*}
  &\ \| W_{\alpha \alpha'}(t,s)k - \widetilde{W}_{\alpha \alpha'}(t,s)k \|_{\alpha''} 
 \\  &\leq \int \limits_s^t \cdots \int \limits_{s}^{t_{n-1}}\| \widetilde{Q}_n(t,t_1,\dots, t_n) \|_{\alpha'' \alpha'} \| W_{\alpha \alpha''}(t_n,s)k - \widetilde{W}_{\alpha \alpha''}(t_n,s)k \|_{\alpha''} dt_n \dots dt_1
 \\ &\leq C(\alpha, \alpha'', k)  K^{n} \frac{M(\alpha'')^n(2n)^n}{(\alpha'' - \alpha')^n} \frac{(t-s)^n}{n!},
 \end{align*}
 where $C(\alpha, \alpha'',k) = \sup_{r \in [s,t]}  \| W_{\alpha \alpha''}(r,s) - \widetilde{W}_{\alpha \alpha''}(r,s) \|_{\alpha \alpha''}$
 is finite by strong continuity and the uniform boundedness principle.
 Since the right-hand side tends to zero as $n \to \infty$, we conclude the assertion. 
 For the last property (d), observe that the sequence $(W_n(t,s)k)_{n \in \N}$ also satisfies the relation
 \[
  W_{n+1}(t,s)k = \int \limits_{s}^{t}W_{n}(t,r)B(r)V(r,s)k dr.
 \]
 A repetition of above arguments proves (d).
\end{proof}
If we suppose, in addition, that also (A1) and (A3) are satisfied, then $W(t,s)$ is continuously differentiable.
Consequently, we are able to solve the corresponding forward evolution equation given in the next definition.
\begin{Definition}
 Fix $\alpha_* < \alpha' < \alpha$, $k \in \E_{\alpha}$ and $s \geq 0$. 
 A solution on $[s,s+T]$, $T > 0$, to the forward evolution equation
  \begin{align}\label{EQ:04}
   \frac{d}{dt} i_{\alpha' \alpha''}u(t) = (A_{\alpha' \alpha''}(t) + B_{\alpha' \alpha''}(t))u(t), \ \ u(s) = k, \ \ t \in [s,s + T]
  \end{align}
  is, by definition, a function $u \in C([s, s+T]; \E_{\alpha'})$, such that $i_{\alpha' \alpha''}u \in C^1([s,s+T]; \E_{\alpha''})$ and \eqref{EQ:04} 
  holds for all $\alpha'' \in (\alpha_*, \alpha')$.
\end{Definition}
Note that the right-hand side in \eqref{EQ:04} is independent of the particular choice of $\alpha'' \in (\alpha_*, \alpha')$.
Concerning existence and uniqueness for \eqref{EQ:04} we obtain the following.
\begin{Corollary}\label{CORR:00}
 Let $(A(t))_{t \geq 0}$ be given as in (A1) and let $(V(t,s))_{t \geq s \geq 0}$ be such that (A2) -- (A4) holds.
 Suppose that $(B(t))_{t \geq 0}$ satisfies (B1), (B2). Let $W(t,s)$ be given by Theorem \ref{THEOREM:00}.
 Then the following assertions hold
 \begin{enumerate}
  \item[(a)] For all $\alpha' < \alpha'' < \alpha$ and $k \in \E_{\alpha}$,
  \[ 
   [s, s + \min\{ T(\alpha', \alpha), T(\alpha'', \alpha)\} ) \ni t \longmapsto W_{\alpha \alpha'}(t,s)k \in \E_{\alpha'}
  \]
  is continuously differentiable with
  \begin{align*}
   \frac{\partial}{\partial t}W_{\alpha \alpha'}(t,s)k = (A_{\alpha'' \alpha'}(t) + B_{\alpha'' \alpha'}(t))W_{\alpha \alpha''}(t,s)k.
  \end{align*}
  \item[(b)] For all $\alpha' < \alpha'' < \alpha$ and $k \in \E_{\alpha}$,
  \[
   (t - \min\{ T(\alpha', \alpha), T(\alpha', \alpha'')\}, t] \ni s \longmapsto W_{\alpha \alpha'}(t,s)k \in \E_{\alpha'}
  \]
  is continuously differentiable with
  \begin{align} \label{LINEAR:13}
 \frac{\partial}{\partial s}W_{\alpha \alpha'}(t,s)k = - W_{\alpha'' \alpha'}(t,s)(A_{\alpha \alpha''}(s) + B_{\alpha \alpha''}(s))k.
  \end{align}
  \item[(c)] Fix $s \geq 0$, $\alpha' < \alpha$ and $k \in \E_{\alpha}$.
  Suppose that there exists $T > 0$ and a function $u \in C([s,s+T]; \E_{\alpha'})$ which is a solution to \eqref{EQ:04}.
  Then $u(t) = W_{\alpha \alpha'}(t,s)k$ holds for any $t$ with $s \leq t < s + \min\{T, T(\alpha', \alpha)\}$.
\end{enumerate}
\end{Corollary}
\begin{proof}
 Assertions (a) and (b) are direct consequences of Theorem \ref{THEOREM:00}.(c) and Theorem \ref{THEOREM:00}.(d) 
 combined with Lemma \ref{LINEARLEMMA:00}. Let us prove (c).

 Define $w(t) := W_{\alpha \alpha'}(t,s)k - u(t)$, where $s \leq t < s + \min\{T, T(\alpha', \alpha)\}$. 
 Then $w(s) = 0$ and it suffices to show that $w = 0$. 
 Since $w$ solves \eqref{EQ:04} it follows that $w$ satisfies,
 for any $t$ with $s \leq t < s + \min\{T, T(\alpha',\alpha)\}$ and $\alpha'' \in (\alpha_*, \alpha')$,
 \[
  i_{\alpha' \alpha''}w(t) = \int \limits_{s}^{t}(V(t,r)B(r))_{\alpha'' \alpha'}w(r) dr.
 \]
 Iterating this equality yields
 \[
  i_{\alpha' \alpha''}w(t) = \int \limits_{s}^{t}\dots \int \limits_{s}^{t_{n-1}}\widetilde{Q}_n(t,t_1,\dots, t_n,s)_{\alpha'' \alpha'}w(t_n)dt_n \dots, dt_1,
 \]
 where $\widetilde{Q}_n(t,t_1, \dots, t_n) := V(t,t_1)B(t_1)\cdots V(t_{n-1}, t_n)B(t_n) \in L(\E_{\alpha'}, \E_{\alpha''})$.
 In order to estimate this integral, we let $\alpha_j := \alpha'' + j \frac{\alpha' - \alpha''}{2n}$, $j \in \{0, \dots, 2n\}$. Then
 \begin{align*}
  \Vert \widetilde{Q}_n(t,t_1, \dots, t_n)\Vert_{\alpha' \alpha''} \leq K^{n} \frac{M(\alpha')^n(2n)^n}{(\alpha' - \alpha'')^n}, \ \ n \geq 1.
 \end{align*}
 Letting $C_{\alpha'} := \sup_{r \in [s,t]}\ \Vert w(r) \Vert_{\alpha'} < \infty$ 
 and using $\Vert w(t_n) \Vert_{\alpha'} \leq C_{\alpha'}$ we obtain
 \begin{align*}
  \Vert w(t)\Vert_{\alpha''} \leq C_{\alpha'} \left( \frac{2eK M(\alpha')(t-s)}{\alpha' - \alpha''}\right)^n.
 \end{align*}
 If, in addition, $0 \leq t - s < T(\alpha'', \alpha')$, then the right-hand side tends to zero as $n \to \infty$ from which we deduce $w(t) = 0$ in $\E_{\alpha''}$ 
 and thus also in $\E_{\alpha}$ for all $t$ satisfying
 \[
  s \leq t < s + \min \left\{ T, T(\alpha', \alpha),  T(\alpha'', \alpha') \right\}.
 \]
 Setting $\alpha'' = \frac{\alpha' + \alpha_*}{2} \in (\alpha_*, \alpha')$ shows that, 
 for any $\alpha' < \alpha$, $k \in \E_{\alpha}$ and $s \geq 0$, equation \eqref{EQ:04} has a unique solution on $[s, s + \frac{1}{2}T_0(\alpha', \alpha)]$ 
 where $T_0(\alpha', \alpha) := \min \{ T(\alpha', \alpha), \frac{\alpha' - \alpha_*}{4eKM(\alpha')}, T \}$. 
 
 Changing $s$ to $s + \frac{1}{2}T_0(\alpha', \alpha)$ and iterating this procedure yields the assertion. 
 Note that such an iteration is possible since
 the new initial condition satisfies $w(s + \frac{1}{2}T_0(\alpha',\alpha)) = 0 \in \E_{\alpha}$.
\end{proof}
\begin{Remark}
 The proofs show that, if $t \longmapsto B_{\alpha \alpha'}(t) \in L(\E_{\alpha}, \E_{\alpha'})$ is continuous in the operator norm for any $\alpha' < \alpha$, 
 then $W_{\alpha \alpha'}(t,s)$ is continuously differentiable in the operator norm on $L(\E_{\alpha}, \E_{\alpha'})$.
\end{Remark}
As a consequence of the uniqueness results we see that $W(t,s)$ satisfies the forward evolution property.
\begin{Corollary}
 Let $(A(t))_{t \geq 0}$ be given as in (A1) and let $(V(t,s))_{t \geq s \geq 0}$ be such that (A2) -- (A4) holds.
 Suppose that $(B(t))_{t \geq 0}$ satisfies (B1), (B2). Let $W(t,s)$ be the evolution system given by Theorem \ref{THEOREM:00}.
 Then, for all $\alpha' < \alpha'' < \alpha$ and $0 \leq s \leq r \leq t$ satisfying the relations 
 \[
  t - s < T(\alpha', \alpha), \qquad r-s < T(\alpha'', \alpha), \qquad t-r < T(\alpha', \alpha''),
 \]
 it holds that $W_{\alpha \alpha'}(t,s) = W_{\alpha'' \alpha'}(t,r) W_{\alpha \alpha''}(r,s)$.
\end{Corollary}
\begin{Remark}
 If $B(t) = 0$ for all $t \geq 0$, then $M(\alpha) = 0$ and hence $T(\alpha', \alpha) = + \infty$.
 Consequently $W_{\alpha \alpha'}(t,s) = V_{\alpha \alpha'}(t,s)$ is defined for all $t \geq s \geq 0$.
\end{Remark}
Below we provide one sufficient condition under which $W(t,s)$ is defined for all $t \geq s \geq 0$.
\begin{Remark}\label{LINEARTH:06}
 Suppose that the same conditions as for Theorem \ref{THEOREM:00} are satisfied, and assume that
 \[
 \sup\limits_{\alpha > \alpha_*}  M(\alpha) =: M^* < \infty.
 \]
 Let $W(t,s)$ be the evolution system given by Theorem \ref{THEOREM:00} with $T(\alpha', \alpha) = \frac{\alpha - \alpha'}{2eKM(\alpha)}$,
 and let $W^0(t,s)$ be the evolution system given by Theorem \ref{THEOREM:00} with $T_0(\alpha', \alpha) = \frac{\alpha - \alpha'}{2eK M^*}$.
 By uniqueness we obtain, for all $\alpha' < \alpha$,
 \[
  W_{\alpha \alpha'}(t,s) = W_{\alpha \alpha'}^0(t,s), \ \ 0 \leq t - s < T_0(\alpha', \alpha) \leq T(\alpha', \alpha).
 \]
 Moreover, for each $T > 0$ and $\alpha' > \alpha_*$ there exists $\alpha = \alpha(T) > \alpha'$ such that $T < T_0(\alpha', \alpha)$,
 i.e. $W_{\alpha \alpha'}(t,s)$ is defined on $[0,T]$.
 This operator is clearly independent of the particular choice of $\alpha(T)$ as long as $T < T_0(\alpha', \alpha)$,
 see also Theorem \ref{THEOREM:00}.(b).
\end{Remark}

\subsection{Stability for the forward evolution system}
In this section we study stability for the constructed evolution systems.
Consider, for each $n \in \N$, the following set of conditions
\begin{enumerate}
 \item[(S1)] Let $(A^{(n)}(t))_{t \geq 0} \subset L(\E)$ be such that
 $t \longmapsto A_{\alpha \alpha'}^{(n)}(t) \in L(\E_{\alpha}, \E_{\alpha'})$, $\alpha' < \alpha$, is continuous in the operator topology. 
 \item[(S2)] Let $(V^{(n)}(t,s))_{t \geq s \geq 0} \subset L(\E)$ satisfy (A2) and (A3) for the operators $(A^{(n)}(t))_{t \geq 0}$. 
 Moreover, suppose that there exists a constant $K \geq 1$ such that
 \begin{align}\label{LINEAR:15}
  \Vert V^{(n)}_{\alpha \alpha'}(t,s)\Vert_{\alpha \alpha'} \leq K, \ \ 0 \leq s \leq t, \ \alpha' < \alpha, \ \ n \in \N.
 \end{align}
 \item[(S3)] Let $(B^{(n)}(t))_{t \geq 0} \subset L(\E)$ be such that
 $t \longmapsto B_{\alpha \alpha'}^{(n)}(t) \in L(\E_{\alpha}, \E_{\alpha'})$, $\alpha' < \alpha$, is strongly continuous. 
 Moreover, suppose that there exists an increasing continuous function 
 $M: (\alpha_*, \infty) \longrightarrow [0,\infty)$ independent of $n$ such that
 \begin{align}\label{LINEAR:16}
  \Vert B^{(n)}_{\alpha \alpha'}(t)\Vert_{\alpha \alpha'} \leq \frac{M(\alpha)}{\alpha - \alpha'}, \ \ \alpha' < \alpha, \ \ t \geq 0.
 \end{align}
\end{enumerate}
Then we obtain the following.
\begin{Theorem}\label{LINEARTH:05}
 Suppose that (S1) -- (S3) are satisfied. 
 Let $(A(t))_{t \geq 0}$ and $(B(t))_{t \geq 0}$ be two families of operators in $L(\E)$ such that, for any $T > 0$ and $\alpha' < \alpha$,
 \begin{align}\label{LINEAR:20}
  \lim \limits_{n \to \infty} \sup \limits_{t \in [0,T]} \Vert B^{(n)}_{\alpha \alpha'}(t) - B_{\alpha \alpha'}(t)\Vert_{\alpha \alpha'} = 0
 \end{align}
 and
 \begin{align}\label{LINEAR:17}
  \lim \limits_{n \to \infty} \sup \limits_{t \in [0,T]} \Vert A^{(n)}_{\alpha \alpha'}(t) - A_{\alpha \alpha'}(t)\Vert_{\alpha \alpha'} = 0.
 \end{align}
 Then $(A(t))_{t \geq 0}$ satisfies (A1), $(B(t))_{t \geq 0}$ satisfies (B1) and (B2) with $M(\alpha)$ being the same as in \eqref{LINEAR:16}, 
 and there exists $(V(t,s))_{t \geq s \geq 0}$ satisfying (A2) -- (A4) with a constant $K$ being the same as in \eqref{LINEAR:15}.
 Moreover, for any $n \in \N$, there exist operators $W^{(n)}(t,s)$ and $W(t,s)$ given by Theorem \ref{THEOREM:00}
 with $T(\alpha', \alpha) = \frac{\alpha - \alpha'}{2eKM(\alpha)}$ and,
 for any $\alpha' < \alpha$ and any compact $\Delta_{\alpha' \alpha} \subset \{ (t,s) \ | \ 0 \leq t-s < T(\alpha', \alpha) \}$, it holds that
 \begin{align}\label{STABILITY:EST}
  \lim \limits_{n \to \infty} \sup \limits_{(t,s) \in \Delta_{\alpha' \alpha}} \| W^{(n)}_{\alpha \alpha'}(t,s)k - W_{\alpha \alpha'}(t,s)k\|_{\alpha'} = 0, \ \ k \in \E_{\alpha}.
 \end{align}
\end{Theorem}
\begin{proof}
 From \eqref{LINEAR:17} it follows that $t \longmapsto A_{\alpha \alpha'}(t)$ is continuous in the operator norm and hence satisfies (A1).
 Analogously, one shows that \eqref{LINEAR:20} implies (B1) and (B2).
 Applying Lemma \ref{LEMMA:STABILITY} to $V^{(n)}(t,s)$ and $V^{(m)}(t,s)$ yields,
 for all $t \geq s \geq 0$ and $\alpha ' < \alpha_0 < \alpha_1 < \alpha$, 
 \[
  \| V^{(n)}_{\alpha \alpha'}(t,s) - \widetilde{V}^{(m)}_{\alpha \alpha'}(t,s) \|_{\alpha \alpha'} 
  \leq C(\alpha_0, \alpha_1, \alpha' , \alpha)K^2 \int \limits_{s}^{t} \| A^{(n)}_{\alpha_1 \alpha_0}(r) - A^{(m)}_{\alpha_1\alpha_0}(r) \|_{\alpha_1 \alpha_0}dr.
 \]
 Hence $(V_{\alpha \alpha'}^{(n)}(t,s))_{n \in \N}$ is a Cauchy sequence in $L(\E_{\alpha}, \E_{\alpha'})$.
 Denote by $V_{\alpha \alpha'}(t,s)$ its limit. Then
 \begin{align}\label{LINEAR:19}
  \lim \limits_{n \to \infty}\sup \limits_{(t,s) \in \Delta_{\alpha' \alpha}} \Vert V^{(n)}_{\alpha \alpha'}(t,s) - V_{\alpha \alpha'}(t,s) \Vert_{\alpha\alpha'} = 0.
 \end{align}
 It is not difficult to show that $(V(t,s))_{t \geq s \geq 0}$ satisfies the properties (A2) -- (A4).
 Hence we may apply Theorem \ref{THEOREM:00} and obtain the existence of $W^{(n)}(t,s), W(t,s)$ given by 
 \begin{align}\label{SERIES}
  W(t,s) = \sum \limits_{j=0}^{\infty}W_j(t,s), \qquad W^{(n)}(t,s) = \sum \limits_{j=0}^{\infty} W_j^{(n)}(t,s),
 \end{align}
 where $W_j(t,s)$ and $W_j^{(n)}(t,s)$ are defined in the proof of Theorem \ref{THEOREM:00}.
 Arguing as in the proof of Theorem \ref{THEOREM:00},
 we obtain from \eqref{LINEAR:15} and \eqref{LINEAR:16} the inequality
 \[
  \| W_j(t,s) \|_{\alpha \alpha'}, \ \Vert W_j^{(n)}(t,s) \Vert_{\alpha \alpha'} \leq \left( \frac{t-s}{T(\alpha', \alpha)}\right)^j K, \ \ n \geq 1, \ \ j \geq 0.
 \]
 This implies that the series in \eqref{SERIES} converges in the strong operator topology on $L(\E_{\alpha}, \E_{\alpha'})$, $\alpha' < \alpha$,
 uniformly with respect to $(t,s) \in \Delta_{\alpha' \alpha}$ and $n \in \N$.
 Thus it suffices to show that, for all $j \geq 0$ and $\alpha' < \alpha$,
 \[
  \lim \limits_{n \to \infty} \sup \limits_{(t,s) \in \Delta_{\alpha' \alpha}} \| W_j^{(n)}(t,s)k - W_j(t,s)k \|_{\alpha'} = 0, \ \ k \in \E_{\alpha}.
 \]
 For $j = 0$ this follows from \eqref{LINEAR:19}. For $j \geq 1$ we proceed by induction using the definition of $W_j^n(t,s), W_j(t,s)$ and \eqref{LINEAR:20}.
\end{proof}
\begin{Remark}
 If, in addition to the conditions of Theorem \ref{LINEARTH:05}, the mapping
 \[
  [0, \infty) \ni t \longmapsto B_{\alpha \alpha'}^{(n)}(t)\in L(\E_{\alpha}, \E_{\alpha'})
 \]
 is continuous in the operator norm for all $\alpha' < \alpha$. Then \eqref{SERIES} converges in the operator norm and hence 
 \eqref{STABILITY:EST} also holds in the operator norm.
 This includes e.g. the time-homogeneous case.
\end{Remark}

\subsection{Construction of backward evolution system}
In this section we give an analogous construction for the corresponding backward evolution system.
Namely, we consider a family of operators $(V(s,t))_{t \geq s \geq 0}$ with the property (A2)$^*$ and we assume that
\begin{enumerate}
 \item[(A4)$^*$]  There exists a constant $K \geq 1$ such that for all $\alpha' < \alpha$
 \begin{align*}
  \Vert V_{\alpha \alpha'}(s,t)\Vert_{\alpha \alpha'} \leq K, \ \ 0 \leq s \leq t.
 \end{align*}
\end{enumerate}
The next statement is proved analogously to Theorem \ref{THEOREM:00}.
\begin{Theorem}\label{LINEARTH:03}
 Suppose that $(V(s,t))_{t \geq s \geq 0}$ satisfies (A2)$^*$, (A4)$^*$ and $(B(t))_{t \geq 0}$ satisfies (B1), (B2).
 Define $T(\alpha', \alpha) := \frac{\alpha - \alpha'}{2Ke M(\alpha)}$, $\alpha' < \alpha$.
 Then there exists a family of operators 
 \[
  \{ W_{\alpha \alpha'}(s,t) \in L(\E_{\alpha}, \E_{\alpha'}) \ | \ \alpha' < \alpha, \ \ 0 \leq t-s < T(\alpha', \alpha) \}
 \]
 with the properties
 \begin{enumerate}
  \item[(a)] For any $\alpha' < \alpha$ and $0 \leq t-s < T(\alpha', \alpha)$ we have $W_{\alpha\alpha'}(s,t) \in L(\E_{\alpha}, \E_{\alpha'})$,
  $(s,t) \longmapsto W_{\alpha \alpha'}(s,t)$ is strongly continuous on $L(\E_{\alpha}, \E_{\alpha'})$ with $W_{\alpha \alpha'}(t,t) = i_{\alpha \alpha'}$, and
  \[
   \Vert W_{\alpha \alpha'}(s,t) \Vert_{\alpha \alpha'} \leq  \frac{T(\alpha', \alpha)}{T(\alpha',\alpha) - (t-s)}K.
  \]
  \item[(b)] For all $\alpha' < \alpha'' < \alpha$ with $0 \leq t - s < \min\{ T(\alpha', \alpha), T(\alpha'', \alpha)\}$ we have
  \[
   W_{\alpha \alpha'}(s,t) = i_{\alpha'' \alpha'}W_{\alpha \alpha''}(s,t),
  \]
  and for  all $\alpha' < \alpha'' < \alpha$ with $0 \leq t - s < \min\{ T(\alpha', \alpha), T(\alpha', \alpha'')\}$ we have
  \[
   W_{\alpha \alpha'}(s,t) = W_{\alpha'' \alpha'}(s,t)i_{\alpha \alpha''}.
  \]
  \item[(c)] For all $\alpha' < \alpha'' < \alpha$ with $0 \leq t - s < \min\{ T(\alpha', \alpha), T(\alpha', \alpha'')\}$ we have
  \begin{align*}
   W_{\alpha \alpha'}(s,t)k = V_{\alpha \alpha'}(s,t)k + \int \limits_{s}^{t}W_{\alpha'' \alpha'}(s,r)(B(r)V(r,t))_{\alpha \alpha''}k dr, \ \ \forall k \in \E_{\alpha}.
  \end{align*}
   Moreover $W(s,t)$ is unique with such property.
  \item[(d)] For all $\alpha' < \alpha'' < \alpha$ with $0 \leq t - s < \min\{ T(\alpha', \alpha), T(\alpha'', \alpha)\}$ we have
  \begin{align*}
   W_{\alpha \alpha'}(s,t)k = V_{\alpha \alpha'}(s,t)k + \int \limits_{s}^{t}(V(s,r)B(r))_{\alpha'' \alpha'}W_{\alpha \alpha''}(r,t)k dr, \ \ \forall k \in \E_{\alpha}.
  \end{align*}
   Moreover $W(s,t)$ is also unique with such property.
 \end{enumerate}
\end{Theorem}
\begin{proof}
 Since the proof is very similar to Theorem \ref{THEOREM:00}, we only sketch the main differences.
 Define a sequence of operators $(W_n(s,t))_{0 \leq s \leq t} \subset L(\E)$ by $W_0(s,t) = V(s,t)$ and 
 \begin{align*}
  W_{n+1}(s,t) := \int \limits_{s}^{t}W_n(s,r)B(r)V(r,t) dr,
 \end{align*}
 where the integrals are, for all $\alpha' < \alpha$, defined in the strong topology on $L(\E_{\alpha}, \E_{\alpha'})$.
 As before we can show that, for any $\alpha' < \alpha$, $n \geq 0$ and $k \in \E_{\alpha}$, the function $W_n(s,t)k \in \E_{\alpha'}$ 
 is continuous in $(s,t)$ and satisfies
 \[
  \Vert W_n(s,t)k\Vert_{\alpha'} \leq \Vert k \Vert_{\alpha} \left(\frac{t-s}{T(\alpha',\alpha)}\right)^nK.
 \]
 From this we readily deduce assertion (a). Assertion (b) is again a direct consequence of 
 \[
   (W_n(s,t))_{\alpha' \alpha} = i_{\alpha'' \alpha'}(W_n(s,t))_{\alpha \alpha''} = (W_n(s,t))_{\alpha'' \alpha'}i_{\alpha \alpha''}.
 \]
 Assertion (c) can be shown in the same way as Theorem \ref{THEOREM:00}.(c).
 For the last property (d), observe that the sequence $(W_n(s,t))_{n \in \N}$ also satisfies the relation
 \[
  W_{n+1}(s,t) = \int \limits_{s}^{t}V(s,r)B(r)W_{n}(r,t) dr
 \]
 from which we may deduce (d).
\end{proof}
Below we relate this evolution family $W(s,t)$ with a backward
Cauchy problem.
\begin{Definition}
 Fix $\alpha_* < \alpha' < \alpha$, $k \in \E_{\alpha}$ and $t \geq 0$. A solution on $[t-T,t]$, $T \in (0,t]$, to the backward Cauchy problem
 \begin{align}\label{EQ:05}
  \frac{d}{ds}i_{\alpha' \alpha''}v(s) = - (A_{\alpha' \alpha''}(s) + B_{\alpha' \alpha''}(s))v(s), \ \ v(t) = k, \ \ s \in [t-T,t]
 \end{align}
 is, by definition, a function $v \in C([0,t]; \E_{\alpha'})$, such that $i_{\alpha' \alpha''}u \in C^1([t-T,t]; \E_{\alpha''})$ and 
\eqref{EQ:05} holds for all $\alpha'' \in (\alpha_*, \alpha')$.
\end{Definition}
\begin{Corollary}
 Let $(A(t))_{t \geq 0}$ be given as in (A1) and let $(V(s,t))_{t \geq s \geq 0}$ be such that (A2)$^*$ -- (A4)$^*$ holds.
 Suppose that $(B(t))_{t \geq 0}$ satisfies (B1), (B2). Let $W(s,t)$ be the evolution system given by Theorem \ref{LINEARTH:03}.
 Then the following assertions hold
 \begin{enumerate}
  \item[(a)] For all $\alpha' < \alpha'' < \alpha$ and $k \in \E_{\alpha}$,
  \[
   [s, s +  \min\{ T(\alpha', \alpha), T(\alpha', \alpha'')\} ) \ni t \longmapsto W_{\alpha \alpha'}(s,t)k \in \E_{\alpha'}
  \]
  is continuously differentiable with
  \begin{align*}
   \frac{\partial}{\partial t}W_{\alpha \alpha'}(s,t)k = W_{\alpha''\alpha'}(s,t)(A_{\alpha \alpha''}(t) + B_{\alpha \alpha''}(t))k.
  \end{align*}
  \item[(b)] For all $\alpha' < \alpha'' < \alpha$ and $k \in \E_{\alpha}$,
  \[
   (t -  \min\{ T(\alpha', \alpha), T(\alpha'', \alpha)\}, t] \ni s \longmapsto W_{\alpha \alpha'}(s,t)k \in \E_{\alpha'}
  \]
  is continuously differentiable with
  \begin{align*}
 \frac{\partial}{\partial s}W_{\alpha \alpha'}(s,t)k = - (A_{\alpha'' \alpha'}(s) + B_{\alpha'' \alpha'}(s))W_{\alpha \alpha''}(s,t)k.
  \end{align*}
  \item[(c)] Fix $t > 0$, $\alpha' < \alpha$ and $k \in \E_{\alpha}$.
  Let $u$ be a solution to \eqref{EQ:05} on $[T-t,t]$, where $T \in (0,t]$.
  Then $u(s) = W_{\alpha \alpha'}(s,t)k$ holds for any $s$ with $t - \min\{T, T(\alpha', \alpha)\} < s \leq t$.
\end{enumerate}
\end{Corollary}
\begin{proof}
 Follows by similar arguments to Corollary \ref{CORR:00}.
\end{proof}
\begin{Remark}
 The proofs show that, if $t \longmapsto B_{\alpha \alpha'}(t) \in L(\E_{\alpha}, \E_{\alpha'})$ is continuous in the operator norm for any $\alpha' < \alpha$, 
 then $W_{\alpha \alpha'}(s,t)$ is continuously differentiable in the operator norm on $L(\E_{\alpha}, \E_{\alpha'})$.
\end{Remark}
As a consequence we can deduce the backward evolution property for $W(s,t)$.
\begin{Corollary}
 Let $(A(t))_{t \geq 0}$ be given as in (A1) and let $(V(s,t))_{t \geq s \geq 0}$ be such that (A2)$^*$ -- (A4)$^*$ holds.
 Suppose that $(B(t))_{t \geq 0}$ satisfies (B1), (B2). Let $W(s,t)$ be the evolution system given by Theorem \ref{LINEARTH:03}.
 Then, for all $\alpha' < \alpha'' < \alpha$ and $0 \leq s \leq r \leq t$ satisfying the relations 
 \[
  t - s < T(\alpha', \alpha), \qquad r-s < T(\alpha'', \alpha), \qquad t-r < T(\alpha', \alpha''),
 \]
 it holds that $W_{\alpha \alpha'}(s,t) = W_{\alpha'' \alpha'}(s,r) W_{\alpha \alpha''}(r,t)$.
\end{Corollary}
\begin{Remark}
 A similar stability result as for the forward evolution system can be also obtained in this case.
\end{Remark}

\subsection{The dual Cauchy problem}
Set $\B_{\alpha} := \E_{\alpha}^*$ with $\| \cdot \|_{\B_{\alpha}} =: \vertiii{\cdot}_{\alpha}$, for $\alpha > \alpha_*$. 
Then, for any $\alpha' < \alpha$, $\B_{\alpha'} \ni \ell \longmapsto \ell|_{\E_{\alpha}} \in \B_{\alpha}$ defines an embedding with
$\vertiii{\ell|_{\E_{\alpha}}}_{\alpha} \leq \vertiii{\ell}_{\alpha'}$. By abuse of notation, we denote this embedding also 
by $i_{\alpha' \alpha}$. Hence $\B = (\B_{\alpha})_{\alpha > \alpha_*}$ is a scale of Banach spaces satisfying
\begin{align}\label{LINEAR:50}
 \B_{\alpha'} \subset \B_{\alpha}, \qquad \vertiii{\cdot}_{\alpha} \leq \vertiii{\cdot}_{\alpha'}, \qquad \alpha' < \alpha.
\end{align}
Bounded linear operators in the scale $\B$, the space $L(\B)$ and the composition of bounded linear operators on $\B$, 
are defined analogously to the scale $\E$. Denote by $\langle k, \ell \rangle_{\alpha} := \ell(k)$, where $k \in \E_{\alpha}$, 
$\ell \in \B_{\alpha}$ and $\alpha > \alpha_*$, the dual pairing between $\E_{\alpha}$ and $\B_{\alpha}$. Then, for all $\alpha' < \alpha$, 
\begin{align*}
 \langle i_{\alpha \alpha'} k, \ell \rangle_{\alpha'} = \langle k, i_{\alpha' \alpha} \ell \rangle_{\alpha}, \ \ k \in \E_{\alpha}, \ \ \ell \in \B_{\alpha'}.
\end{align*}
We omit the subscript $\alpha$, if no confusion may arise.
For given $Q \in L(\E)$ the adjoint operator $Q^* \in L(\B)$ is defined by
\[
 (Q_{\alpha' \alpha}^*\ell)(k) := \ell(Q_{\alpha \alpha'}k), \ \  k \in \E_{\alpha}, \ \ \ell \in \B_{\alpha'}, \ \ \alpha' < \alpha,
\]
and hence satisfies $\langle Q_{\alpha \alpha'}k, \ell \rangle = \langle k, Q_{\alpha' \alpha}^* \ell \rangle$ 
for $k \in \E_{\alpha}$ and $\ell \in \B_{\alpha'}$.
\begin{Definition}
 Let $(A(t))_{t \geq 0}$ be given as in (A1) and let 
 $(B(t))_{t \geq 0}$ satisfy (B1) and (B2).
 Let $\mathcal{Y} \subset \bigcap_{\beta > \alpha_*}\E_{\beta}$ be such that $\mathcal{Y}$ is dense in $\E_{\beta}$ 
 for each $\beta > \alpha_*$.
 Fix $\alpha' < \alpha$ and $\ell \in \B_{\alpha'}$.
 \begin{enumerate}
  \item[(a)] Fix $t > 0$ and $T \in (0,t]$. A solution to 
  \begin{align}\label{EQ:12}
   \langle k, \ell(s) \rangle = \langle k, \ell \rangle + \int \limits_{s}^{t} \langle ( A(r) + B(r))_{\alpha'' \alpha}k, \ell(r) \rangle dr,
  \ \ s \in [T-t,t]
  \end{align}
  is, by definition, a family $(\ell(s))_{s \in [T-t,t]} \subset \B_{\alpha}$ such that 
 \begin{enumerate}
  \item[(i)] $[T-t,t] \ni s \longmapsto \langle k, \ell(s) \rangle$ is continuous for any $k \in \E_{\alpha}$.
  \item[(ii)] \eqref{EQ:12} holds for any $k \in \mathcal{Y}$ and all $\alpha'' > \alpha$.
 \end{enumerate}
  \item[(b)] Fix $s \geq 0$ and $T > 0$. A solution to
  \begin{align}\label{EQ:13}
   \langle k, \ell(t) \rangle = \langle k, \ell \rangle + \int \limits_{s}^{t} \langle  ( A(r) + B(r))_{\alpha'' \alpha}k, \ell(r) \rangle dr,
  \ \ t \in [s,s+T]
  \end{align}
  is, by definition, a family $(\ell(t))_{t \in [s,s+T]} \subset \B_{\alpha}$ such that 
 \begin{enumerate}
  \item[(i)] $[s,s+T] \ni t \longmapsto \langle k, \ell(t) \rangle$ is continuous for any $k \in \E_{\alpha}$.
  \item[(ii)] \eqref{EQ:13} holds for any $k \in \mathcal{Y}$ and all $\alpha'' > \alpha$.
 \end{enumerate}
 \end{enumerate}
\end{Definition}
Note that, as before, the right-hand sides in \eqref{EQ:12} and \eqref{EQ:13} are independent of the particular choice of $\alpha'' > \alpha$.
The following is our main result for the dual Cauchy problems.
\begin{Theorem}\label{LINEARTH:11}
 Let $(A(t))_{t \geq 0}$ be given as in (A1) and let 
 $(B(t))_{t \geq 0}$ satisfy (B1) and (B2).
 Let $\mathcal{Y} \subset \bigcap_{\beta > \alpha_*} \E_{\beta}$ be such that $\mathcal{Y}$ is dense in $\E_{\alpha}$ is dense for any $\alpha > \alpha_*$.
 Fix $\alpha' < \alpha$ and $\ell \in \B_{\alpha'}$. Then the following assertions hold:
 \begin{enumerate}
  \item[(a)] Let $t > 0$, $T \in (0,t]$, and suppose that (A2) -- (A4) are satisfied. 
  Then each solution $(\ell(s))_{t - T \leq s \leq t} \subset \B_{\alpha}$ to \eqref{EQ:12} satisfies
  \[
   \ell(s) = W_{\alpha' \alpha}(t,s)^*\ell, \ \ \forall s \text{ with } t - \min\{T, T(\alpha',\alpha)\} < s \leq t.
  \]
 Moreover, for all $k \in \E_{\alpha}$ and $\alpha'' \in (\alpha', \alpha)$, the function
 \[
   (t - \min\{ T(\alpha', \alpha), T(\alpha', \alpha'')\} , t] \ni s \longmapsto \langle k, W_{\alpha' \alpha}(t,s)^* \ell \rangle
 \]
 is continuously differentiable with
 \[
  \frac{d}{ds} \langle k, W_{\alpha' \alpha}(t,s)^* \ell \rangle = - \langle (A_{\alpha'' \alpha}(s) + B_{\alpha \alpha''}(s))k, W_{\alpha' \alpha''}(t,s)^* \ell \rangle.
 \]
  \item[(b)] Let $s \geq 0$ and $T > 0$. Suppose that (A2)$^*$ -- (A4)$^*$ are satisfied. 
 Then each solution $(\ell(t))_{t \in [s, s+T]} \subset \B_{\alpha}$ to \eqref{EQ:13} satisfies
 \[
  \ell(t) = W_{\alpha' \alpha}(s,t)^* \ell, \ \ \forall t \text{ with } s \leq t < s + \min\{T, T(\alpha', \alpha)\}.
 \]
 Moreover, for all $k \in \E_{\alpha}$ and $\alpha'' \in (\alpha', \alpha)$, the function
 \[
  [s, s + \min\{ T(\alpha', \alpha), T(\alpha', \alpha'')\}) \ni t \longmapsto \langle k, W_{\alpha' \alpha}(s,t)^* \ell \rangle
 \]
 is continuously differentiable with
 \[
  \frac{d}{dt} \langle k, W_{\alpha' \alpha}(s,t)^* \ell \rangle = \langle (A_{\alpha \alpha''}(t) + B_{\alpha \alpha''}(t))k, W_{\alpha' \alpha''}(s,t)^* \ell \rangle.
 \]
 \end{enumerate}
\end{Theorem}
\begin{proof}
 We will prove only assertion (a). Assertion (b) can be deduced by similar arguments.
 It is clear that $\ell(s)$ given by $\ell(s) = W_{\alpha' \alpha}(t,s)^* \ell$ is, by duality and Corollary \ref{CORR:00}, 
 a solution to \eqref{EQ:12}.
 Let us show that it is the only solution. Take any solution $(\ell(s))_{t - T \leq s \leq t} \subset \B_{\alpha}$  to \eqref{EQ:12}. 
 Then, for all $\alpha'' > \alpha$ and $k \in \E_{\alpha''}$, $s \longmapsto \langle i_{\alpha'' \alpha}k, \ell(s)\rangle$ is continuously differentiable with
 \begin{align}\label{EQ:20}
  \frac{d}{ds} \langle i_{\alpha'' \alpha}k, \ell(s) \rangle = - \langle (A_{\alpha'' \alpha}(s) + B_{\alpha'' \alpha}(s))k, \ell(s) \rangle.
 \end{align}
 Indeed, let $(k_n)_{n \in \N} \subset \mathcal{Y}$ be such that $\| k_n - k\|_{\alpha''} \longrightarrow 0$ as $n \to \infty$. Then
 \begin{align*}
  \langle i_{\alpha'' \alpha}k_n, \ell(s) \rangle = \langle i_{\alpha'' \alpha'}k_n, \ell \rangle
  + \int \limits_{s}^{t} \langle (A_{\alpha'' \alpha}(r) + B_{\alpha'' \alpha}(r))_{\alpha'' \alpha}k_n, \ell(r) \rangle dr, \ \ n \geq 1
 \end{align*}
 and one has
 \[
  \|  (A_{\alpha'' \alpha}(r) + B_{\alpha'' \alpha}(r))_{\alpha'' \alpha}k_n \|_{\alpha}
  \leq \sup \limits_{n \geq 1} \| k_n \|_{\alpha''} \sup \limits_{r \in [s,t]} \|  (A_{\alpha'' \alpha}(r) + B_{\alpha'' \alpha}(r))_{\alpha'' \alpha} \|_{\alpha'' \alpha} < \infty.
 \]
 Since $(A_{\alpha'' \alpha}(r) + B_{\alpha'' \alpha}(r))_{\alpha'' \alpha} \in L(\E_{\alpha''}, \E_{\alpha})$, we can take the limit $n \to \infty$
 to deduce that, for any $s \in [t-T,t]$,
 \begin{align*}
  \langle i_{\alpha'' \alpha}k, \ell(s) \rangle = \langle i_{\alpha'' \alpha'}k,  \ell \rangle
  + \int \limits_{s}^{t} \langle (A_{\alpha'' \alpha}(r) + B_{\alpha'' \alpha}(r))_{\alpha'' \alpha}k, \ell(r) \rangle dr,
 \end{align*}
 i.e. \eqref{EQ:20} is satisfied.

 Define $w(s) := \ell(s) - W_{\alpha' \alpha}(t,s)^*\ell$. Then, for all $\alpha'' > \alpha$ and $k \in \E_{\alpha''}$, it follows that
 \begin{align*}
  \langle k, w(s) \rangle &=  \int \limits_{s}^{t} \langle \left( A(r) + B(r)\right)_{\alpha'' \alpha}k, w(r) \rangle dr
 \\ &=  \int \limits_{s}^{t}\cdots \int \limits_{s}^{t_{n-1}}\langle Q_n(t,t_1,\dots,t_n)_{\alpha'' \alpha}k, w(r)\rangle dt_n \dots dt_1,
 \end{align*}
 where $Q_n(t,t_1,\dots, t_n)_{\alpha'' \alpha} =  \left(  (A(t_n) + B(t_n)) \cdots (A(t_1) + B(t_1))\right)_{\alpha'' \alpha}$
 can be estimated in the same way as in Theorem \ref{THEOREM:00}, i.e.
 \[
  \| Q_n(t,t_1,\dots, t_n)_{\alpha'' \alpha} \|_{\alpha'' \alpha} \leq \left(\frac{2n K M(\alpha'')}{\alpha'' - \alpha}\right)^n, \ \ n \geq 1.
 \]
 Hence we obtain
 \begin{align*}
  | \langle k, w(s) \rangle | &\leq \int \limits_{s}^{t}\cdots \int \limits_{s}^{t_{n-1}} \| Q_n(t,t_1,\dots, t_n)_{\alpha'' \alpha}\|_{\alpha'' \alpha} \| k \|_{\alpha''} \sup \limits_{r \in [s,t]}\| w(r) \|_{\alpha}dt_n \dots dt_1
 \\ &\leq \| k \|_{\alpha''} \sup \limits_{r \in [s,t]} \| w(r) \|_{\alpha} \left(\frac{2n K M(\alpha'')}{\alpha'' - \alpha}\right)^n \frac{(t-s)^n}{n!}.
 \end{align*}
 Since for $0 \leq t - s < \min\{ T, T(\alpha', \alpha), T(\alpha, \alpha'')\}$ the right-hand side tends to zero as $n \to \infty$, we conclude that
 $w(s) = 0$. Letting $\alpha'' = \alpha + 1$ yields for $T_0(\alpha', \alpha) = \min \left\{T, T(\alpha', \alpha), \frac{1}{2eKM(\alpha+1)} \right\}$
 \[
 w(s) = 0, \ \ t - \frac{1}{2}T_0(\alpha', \alpha) \leq s \leq t.
 \]
 Since $w(t - \frac{1}{2}T_0(\alpha', \alpha)) = 0 \in \B_{\alpha'}$, we may iterate this argument by changing $t \longmapsto t - \frac{1}{2}T_0(\alpha', \alpha)$.
 This proves the assertion.
\end{proof}
Let us close this section with the example considered before.
\begin{Remark}
 Let $\E = (\E_{\alpha})_{\alpha > \alpha_*}$ be a scale of Banach spaces and let, for each $\alpha > \alpha_*$,
 $(T_{\alpha}(t))_{t \geq 0}$ be a strongly continuous
 semigroup with generator $(A_{\alpha}, D(A_{\alpha}))$ satisfying
 properties (i) and (ii) of Remark \ref{TIMEHOMOGENEOUS}.
 Let $B(t)$ be given with properties (B1) and (B2).
 Then all results of this section are applicable to this case.
\end{Remark}

\section{Infinite system of ordinary differential equations}
Let $(a_{nk})_{n,k = 0}^{\infty}$ be an infinite matrix having complex-valued entries and $x = (x_n)_{n=0}^{\infty}$ be the initial condition. 
We apply our results to the infinite system of ordinary differential equations 
\begin{align}\label{LINEAR:47}
 \frac{d u_n(t)}{dt} = \sum \limits_{k=0}^{\infty}a_{nk}u_k(t), \ \ u_n(0) = x_n, \ \ n \in \N_0,
\end{align}
where $u(t) = (u_n(t))_{n = 0}^{\infty}$ is a sequence of complex numbers.
For $\alpha \in \R$ let
\[
 \E_{\alpha} := \left \{ u = (u_n)_{n=0}^{\infty} \ | \ \| u \|_{\alpha} := \sum \limits_{n=0}^{\infty}|u_n| e^{\alpha n} < \infty \right\}.
\]
be the Banach space of all complex-valued sequences with norm $\| \cdot \|_{\alpha}$.
Then $\E = (\E_{\alpha})_{\alpha}$ defines a scale of Banach spaces in the sense of \eqref{LINEAR:29}.
\begin{Remark}
 Let $B$ be a linear mapping given by
 \[
  (Bu)_n := \sum \limits_{k = 0}^{\infty}b_{nk}u_k, \ \ n \in \N_0,
 \]
 provided $u = (u_n)_{n = 0}^{\infty}$ is such that the series is absolutely convergent.
 Then $B \in L(\E)$ iff one has
 \[
  \sup \limits_{k \geq 0} e^{- \alpha k} \sum \limits_{n=0}^{\infty}|b_{nk}|e^{\alpha' n} < \infty, \ \ \alpha' < \alpha.
 \]
 In such a case it holds that $B_{\alpha \alpha'}u = Bu$ and 
 $\| B_{\alpha \alpha'}\| = \sup_{k \geq 0} e^{- \alpha k} \sum_{n=0}^{\infty}|b_{nk}|e^{\alpha' n}$. 
\end{Remark}
We study \eqref{LINEAR:47} under the following conditions
\begin{enumerate}
 \item[(E1)] There exist $(d_n)_{n \in \N_0} \subset (0,\infty)$ and $(b_{nk})_{n,k \in \N_0}, (c_{nk})_{n,k \in \N_0} \subset \C$ such that
\[
 a_{nk} = - \delta_{nn}d_{n} + b_{nk} + c_{nk}, \ \ n,k \in \N_0,
\]
where $\delta_{nk}$ denotes the Kronecker-delta symbol.
\item[(E2)] There exists $\alpha_* \in \R$ and, for all $\alpha > \alpha_*$, a constant $q(\alpha) \in (0,1)$ with
 \[
  e^{-\alpha k}\sum \limits_{n=0}^{\infty}|b_{nk}|e^{\alpha n} \leq q(\alpha)d_k, \ \ k \in \N_0, \ \ \alpha > \alpha_*.
 \]
 \item[(E3)] We have $\sup_{n \in \N}\ d_n e^{- \nu n} < \infty$ for all $\nu > 0$.
 \item[(E4)] There exists a continuous increasing function $M: (\alpha_*, \infty) \longrightarrow (0,\infty)$ such that
 \[
  e^{-\alpha k}\sum \limits_{n=0}^{\infty}|c_{nk}|e^{\alpha' n} \leq \frac{M(\alpha)}{\alpha - \alpha'}, \ \ k \in \N_0, \ \ \alpha' < \alpha.
 \]
\end{enumerate}
Let us show that under conditions (E1) -- (E4), equation \eqref{LINEAR:47} is a particular case of the results obtained in Section 3.
\begin{Theorem}
 Suppose that (E1) -- (E4) are satisfied. 
 Then, for all $\alpha' < \alpha$ and $x \in \E_{\alpha}$, 
 there exists a unique classical solution $(u(t))_{t \in [0,T(\alpha', \alpha))}$ in $\E_{\alpha'}$ with $T(\alpha', \alpha) = \frac{\alpha - \alpha'}{2eM(\alpha)}$
 to \eqref{LINEAR:47}.
\end{Theorem}
\begin{proof}
 Define linear mappings $A,B,C$ by
 \begin{align*}
  (Au)_n = - d_n u_n, \ \ (Bu)_n = \sum \limits_{k=0}^{\infty}b_{nk}u_k,\ \ (Cu)_n = \sum \limits_{k=0}^{\infty}c_{nk}u_k,
 \end{align*}
 where $n \in \N_0$ and $u = (u_n)_{n=0}^{\infty}$ is such that the sums are absolutely convergent.
 In view of (E1), \eqref{LINEAR:47} is equivalent to
 \[
  \frac{du_n(t)}{dt} = (Au(t))_n + (Bu(t))_n + (Cu(t))_n, \ \ u_n(0) = x_n, \ \ n \in \N_0.
 \]
 Hence it suffices to show that Theorem \ref{THEOREM:00} is applicable.

 For $\alpha \in (\alpha_*, \alpha^*)$ let $(U_{\alpha}(t)u)_n := e^{- t d_n}u_n$.
 Then $(U_{\alpha}(t))_{t \geq 0}$ is a holomorphic, positive semigroup with generator $(A_{\alpha}, D(A_{\alpha}))$ given by
 \[
  (A_{\alpha}u)_n = (Au)_n = - d_n u_n, \qquad u \in D(A_{\alpha}) = \{ v \in \E_{\alpha} \ | \ (d_n u_n)_{n = 0}^{\infty} \in \E_{\alpha} \}.
 \]
 Define $(B_{\alpha}, D(A_{\alpha}))$ and $(B_{\alpha}', D(A_{\alpha}))$ by 
 \[
  (B_{\alpha}u)_n = (Bu)_n = \sum \limits_{k = 0}^{\infty}b_{nk}u_k, \qquad
  (B_{\alpha}'u)_n = \sum \limits_{k=0}^{\infty}|b_{nk}|u_k, \ \ u \in D(A_{\alpha}).
 \]
 Then $(B_{\alpha}', D(A_{\alpha}))$ is a well-defined positive operator satisfying
 \begin{align*}
  \| B_{\alpha}' u\|_{\alpha} \leq \sum \limits_{k=0}^{\infty}e^{\alpha k} \left( \sum \limits_{n=0}^{\infty}|b_{nk}|e^{\alpha n} \right)e^{- \alpha k}|u_k|
  \leq q(\alpha) \| A_{\alpha} u \|_{\alpha}, \ \ u \in D(A_{\alpha}).
 \end{align*}
 Hence by \cite[Theorem 2.2]{TV06} and \cite[Theorem 1.1]{AR91} it follows that $(A_{\alpha} + B'_{\alpha},  D(A_{\alpha}))$ 
 is the generator of a holomorphic contraction semigroup on $\E_{\alpha}$. 
 Applying \cite[Theorem 1.2]{AR91} it follows that also $(A_{\alpha} + B_{\alpha}, D(A_{\alpha}))$ is the generator
 of a holomorphic contraction semigroup $(T_{\alpha}(t))_{t \geq 0}$ on $\E_{\alpha}$. 
 Then $(T_{\alpha}(t))_{t \geq 0}$ with generator $(A_{\alpha} + B_{\alpha}, D(A_{\alpha}))$ satisfies the conditions
 of Remark \ref{TIMEHOMOGENEOUS}.
 Moreover, one has $A, B, C \in L(\E)$ and $C$ satisfies conditions (B1) and (B2).
 Hence Theorem \ref{THEOREM:00} is applicable.
\end{proof}
The next statement shows that the unique solution to \eqref{LINEAR:47} can be appxoimated by solutions $u^N$ to certain finite-dimensional
ordinary differential equations. Such result may be useful for numerical simulations.
\begin{Theorem}
 Suppose that (E1) -- (E4) are satisfied. Define, for each $N \geq 1$, a new sequence $(a_{nk}^N)_{n,k=0}^{\infty}$ 
 via $a^N_{nk} = - \delta_{nn}d_n^N + b^N_{nk} + c^N_{nk}$, by setting 
 \[
  d_n^N = \1_{\{n \leq N\}} d_n, \qquad b^N_{nk} = \1_{ \{ n,k \leq N\} }b_{nk}, \qquad c_{nk}^N = \1_{\{ n,k \leq N\}} c_{nk}.
 \]
 Then, for all $\alpha' < \alpha$ and $x \in \E_{\alpha}$, there exists a unique classical solution $(u(t))_{0 \leq t < T(\alpha', \alpha)} \subset \E_{\alpha'}$
 to \eqref{LINEAR:47} with $T(\alpha', \alpha) = \frac{\alpha - \alpha'}{2eM(\alpha)}$, and for all $N \geq 1$ there exist unique classical solutions 
 $(u^N(t))_{0 \leq t < T(\alpha', \alpha)} \subset \E_{\alpha'}$ to
 \[
  \frac{d u_n^N(t)}{dt} = \sum \limits_{k = 0}^{\infty}a_{nk}^N u_k(t), \ \ u^N_n(0) = x_n, \ \ n \in \N_0.
 \]
 Moreover it holds that
 \[
  \lim \limits_{N \to \infty} \sup \limits_{t \in [0,T]} \sum \limits_{n = 0}^{\infty}|u_n^N(t) - u_n(t)| e^{\alpha n} = 0,  \ \ T \in (0, T(\alpha', \alpha)).
 \]
\end{Theorem}
\begin{proof}
 Apply Theorem \ref{LINEARTH:05}.
\end{proof}
\begin{Remark}
 Using the results of Section 3, we are also able to prove the existence of solutions to the adjoint equation
 \[
  \frac{dv_n(t)}{dt} = \sum \limits_{k = 0}^{\infty}a_{kn}v_n(t), \ \ v_n(0) = x_n, \ \ n \in \N_0,
 \]
  in the dual scale given by a weighted $\ell^{\infty}$ space.
 Uniqueness holds in such a case among all component-wise solutions.
\end{Remark}
 It is not difficult to adapt such arguments to systems of Banach-space valued differential equations. 
 Such equations arise naturally from the analysis of spatial birth-and-death processes, see, e.g., \cite{BKKK13, KK18}.

\section{The spatial logistic model in the continuum}

\subsection{Description of the model}
Let $\Gamma$ be the space of all locally finite subsets of $\R^d$, see \eqref{GAMMA}.
We endow $\Gamma$ with the smallest topology such that, for any continuous function 
$f: \R^d \longrightarrow \R$ having compact support, $\Gamma \ni \gamma \longmapsto \sum_{x \in \gamma}f(x)$ is continuous. 
Then $\Gamma$ is a Polish space, see, e.g., \cite{KK06}. 
The space of finite configurations is defined by
\[
 \Gamma_0 = \bigsqcup \limits_{n=0}^{\infty}\Gamma_0^{(n)}, 
 \ \ \Gamma_0^{(n)} = \{ \eta \subset \R^d \ | \ |\eta| = n \}, \ \ \Gamma_0^{(0)} = \{ \emptyset\}.
\]
We endow $\Gamma_0$ with the $\sigma$-algebra generated by 
cylinder sets of the form
\[
 \{ \eta \in \Gamma_0 \ | \ |\eta \cap \Lambda| = n \}, \qquad n \geq 0, \ \ \Lambda \subset \R^d \ \text{ compact. }
\]
Let $B_{bs}(\Gamma_0)$ be the space of all functions $G: \Gamma_0 \longrightarrow \R$ such that 
\begin{enumerate}
 \item[(i)] $G$ is bounded and measurable.
 \item[(ii)] There exists $N(G) \in \N$ and a compact $\Lambda(G) \subset \R^d$ with $G(\eta) = 0$, whenever $|\eta| > N(G)$
 or $\eta \cap \Lambda^c \neq  \emptyset$.
\end{enumerate}
The space of all polynomially bounded cylinder functions is defined by
\[
 \mathcal{FP}(\Gamma) = \left\{ F: \Gamma \longrightarrow \R \ | \ \exists G \in B_{bs}(\Gamma_0) \text{ with } F(\gamma) = \sum \limits_{\eta \Subset \gamma}G(\eta) \ \ \forall \gamma \in \Gamma \right\},
\]
where $\Subset$ indicates that the sum runs only over all finite subsets of $\gamma$.
\begin{Remark}\label{REMARK}
 For each $F \in \mathcal{FP}(\Gamma)$ there exists $N(G) \in \N$, $A(G) \geq 0$ and a compact $\Lambda(G) \subset \R^d$ such that 
 $F(\gamma) = F(\gamma \cap \Lambda)$ and
 $|F(\gamma)| \leq A(1+ |\gamma \cap \Lambda|)^N$.
\end{Remark}
Consider a (heuristic) Markov operator $L$ acting on $\mathcal{FP}(\Gamma)$ via the formula
\begin{align*}
 (LF)(\gamma) &= \sum \limits_{x \in \gamma}\left( m + \sum \limits_{y \in \gamma \backslash x}a^-(x-y)\right)(F(\gamma \backslash x) - F(\gamma))
 \\ &\ \ \ + \sum \limits_{x \in \gamma}\int \limits_{\R^d}a^+(x-y)(F(\gamma^+ \cup y) - F(\gamma))dy.
\end{align*}
For simplicity of notation, we have let $\gamma \backslash x$ and $\gamma \cup x$ stand for $\gamma \backslash \{x \}$ and $\gamma \cup \{x\}$.
The first term describes the death of a particle located at $x \in \gamma$.
Such a Markov event may be either caused by a state-independent mortality with parameter $m \geq 0$,
or by an interaction with another particle at position $y \in \gamma \backslash x$ at rate $a^-(x-y)\geq 0$. 
The second term describes the branching mechanism where each particle at position $x \in \gamma$ may create a new particle at position $y \in \R^d$.
The distribution and rate of such a Markov event is described by $a^+(x-y) \geq 0$.

Using Remark \ref{REMARK} it is easily seen that $LF(\gamma)$
is well-defined for any $\gamma \in \Gamma$ and $F \in \mathcal{FP}(\Gamma)$, provided that $a^{\pm}$ are bounded and have compact support.
Such model was studied in \cite{FKK09, KK18} where the following balance condition for the birth-and-death rates has been used
\begin{enumerate}
 \item[(G)] $a^{\pm} \geq 0$ are symmetric, bounded, integrable and that there exists $\vartheta > 0$ and $b \geq 0$ such that
 \begin{align}\label{STABILITY}
  \sum \limits_{x \in \eta}\sum \limits_{y \in \eta \backslash x}\left( a^-(x-y) - \vartheta a^+(x-y)\right) \geq - b |\eta|, \ \ \eta \in \Gamma_0.
 \end{align}
\end{enumerate}
Note that this condition does not imply that $LF(\gamma)$ is well-defined for all $\gamma \in \Gamma$. 
Hence we define $LF$ in a different way explained below.

\subsection{The Fokker-Planck equation}
In the formulation of the corresponding Fokker-Planck equation \eqref{FPE} we will only require that $LF$ is well-defined
for $\mu$-a.a $\gamma \in \Gamma$, where $\mu$ belongs to a certain class of states on $\Gamma$.
Hence we do not assume that $a^{\pm}$ have compact supports.
\begin{Definition}
 Let $\mu_0$ be a Borel probability measure on $\Gamma$. 
 A solution to \eqref{FPE} is a family of Borel probability measures $(\mu_t)_{t \geq 0}$ on $\Gamma$ satisfying,
 for any $F \in \mathcal{FP}(\Gamma)$,
 \begin{enumerate}
  \item[(a)] $F, LF \in L^1(\Gamma, \mu_t)$ for all $t \geq 0$.
  \item[(b)] $t \longmapsto \int_{\Gamma}(LF)(\gamma)d\mu_t(\gamma)$ is locally integrable and  
  \[
   \int \limits_{\Gamma}F(\gamma)d\mu_t(\gamma) = \int \limits_{\Gamma} F(\gamma)d\mu_0(\gamma) + \int \limits_{0}^{t}\int \limits_{\Gamma}(LF)(\gamma)d\mu_s(\gamma) ds, \qquad t \geq 0.
  \]
 \end{enumerate}
\end{Definition}
Below we study the Fokker-Planck equation in terms of the corresponding correlation function evolution obtained from \eqref{BBGKY}.
Namely, introduce the combinatorial transformation
\[
 (KG)(\gamma) := \sum \limits_{\eta \Subset \gamma}G(\eta), \qquad\gamma \in \Gamma, \ G \in B_{bs}(\Gamma_0).
\]
with inverse given by $(K^{-1}F)(\eta) = \sum_{\xi \subset \eta}(-1)^{|\eta \backslash \xi|}F(\xi)$. 
For a Borel probability measure $\mu$ on $\Gamma$ the correlation function $k_{\mu}:\Gamma_0 \longrightarrow \R_+$ is uniquely determined by
\begin{align}\label{DEF:CORR}
 \int \limits_{\Gamma}(KG)(\gamma)d\mu(\gamma) = \int \limits_{\Gamma_0}G(\eta)k_{\mu}(\eta)d\lambda(\eta), \qquad G \in B_{bs}(\Gamma_0),
\end{align}
where $\lambda$ denotes the Lebesgue-Poisson measure on $\Gamma_0$ defined by the relation
\[
 \int \limits_{\Gamma_0}G(\eta)d\lambda(\eta) = G(\{\emptyset\}) + \sum \limits_{n=1}^{\infty}\frac{1}{n!}\int \limits_{(\R^d)^n}G(\{x_1,\dots, x_n\})dx_1\dots dx_n, \ \ G \in B_{bs}(\Gamma_0).
\]
Note that such correlation function does not need to exist.
However, it is necessary and sufficient that $\mu$ has locally finite moments and is locally
absolutely continuous with respect to the Poisson measure on $\Gamma$, see \cite{KK02} and the references therein.

Let us explain how \eqref{BBGKY} can be derived from the Fokker-Planck equation, see \cite{FKO07} for additional details.
Take $\mu$ such that it has correlation function $k_{\mu}$ and let $G \in B_{bs}(\Gamma_0)$. Then, at least formally, 
one obtains for $\widehat{L} := K^{-1}LK$ the relation
\begin{align}\label{DUALITY:01}
 \int \limits_{\Gamma}(LKG)(\gamma)d\mu(\gamma) 
 \int \limits_{\Gamma}(K\widehat{L}G)(\gamma)d\mu(\gamma)
 = \int \limits_{\Gamma_0}(\widehat{L}G)(\eta)k_{\mu}(\eta)d \lambda(\eta).
\end{align}
A computation shows that $\widehat{L}$ is given by $\widehat{L} = \widehat{L}_0 + \widehat{L}_1$ with
\begin{align*}
 (\widehat{L}_0G)(\eta) &= - \left( m |\eta| + \sum \limits_{x \in \eta}\sum \limits_{y \in \eta \backslash x}a^-(x-y)\right)G(\eta)
 + \sum \limits_{x \in \eta}\int \limits_{\R^d}a^+(x-y)G(\eta \cup y)d y,
 \\ (\widehat{L}_1G)(\eta) &= - \sum \limits_{x \in \eta}\sum \limits_{y \in \eta \backslash x}a^-(x-y)G(\eta \backslash x)
 + \sum \limits_{x \in \eta}\int \limits_{\R^d}a^+(x-y)G(\eta \backslash x \cup y)dy.
\end{align*}
Hence using \eqref{DEF:CORR} and \eqref{DUALITY:01} we may reformulate \eqref{FPE} to
\begin{align}\label{WCORR}
 \frac{d}{dt}\int \limits_{\Gamma_0}G(\eta)k_{\mu_t}(\eta)d\lambda(\eta) = \int \limits_{\Gamma_0}(\widehat{L}G)(\eta)k_{\mu_t}(\eta)d\lambda(\eta), \ \ G \in B_{bs}(\Gamma_0).
\end{align}
The operator $L^{\Delta}$ introduced in Section 1 should therefore be related with $\widehat{L}$ by 
\begin{align}\label{DUALITY}
 \int \limits_{\Gamma_0}(\widehat{L}G)(\eta)k(\eta) d\lambda(\eta) = \int \limits_{\Gamma_0}G(\eta)(L^{\Delta}k)(\eta)d\lambda(\eta), \ \ G,k \in B_{bs}(\Gamma_0).
\end{align}
It can be shown that it is given by $L^{\Delta} = L_0^{\Delta} + L_1^{\Delta}$, where
\begin{align*}
 (L_0^{\Delta}k)(\eta) &= - \left( m |\eta| + \sum \limits_{x \in \eta}\sum \limits_{y \in \eta \backslash x}a^-(x-y)\right)k(\eta)
     + \sum \limits_{x \in \eta}\sum \limits_{y \in \eta \backslash x}a^+(x-y)k(\eta \backslash x)
 \\ (L_1^{\Delta}k)(\eta) &= - \sum \limits_{x \in \eta}\int \limits_{\R^d}a^-(x-y)k(\eta \cup y)dy 
 + \int \limits_{x \in \eta}\int \limits_{\R^d}a^+(x-y)k(\eta \backslash x \cup y)dy.
\end{align*}
Hence \eqref{WCORR} is simply a weak formulation to \eqref{BBGKY}.

\subsection{Uniqueness for the Fokker-Planck equation}
Below we introduce corresponding Banach spaces and study uniqueness for \eqref{FPE} and \eqref{WCORR}.
Let $\Lb_{\alpha}$ be the Banach space of functions $G$ equipped with the norm
\[
 \| G\|_{\Lb_{\alpha}} = \int \limits_{\Gamma_0}|G(\eta)|e^{\alpha |\eta|} d\lambda(\eta).
\]
Using the duality $\langle G, k \rangle = \int_{\Gamma_0} G(\eta)k(\eta)d\lambda(\eta)$
we may identify $\Lb_{\alpha}^*$ with the Banach space $\K_{\alpha}$ of functions $k$ equipped with the norm
\[
 \| k \|_{\K_{\alpha}} = \esssup \limits_{\eta \in \Gamma_0} |k(\eta)|e^{- \alpha |\eta|}.
\]
Observe that $\Lb = (\Lb_{\alpha})_{\alpha}$ defines a scale of Banach spaces in the sense of \eqref{LINEAR:29}
while $\K = (\K_{\alpha})_{\alpha}$ corresponds to \eqref{LINEAR:50}.
Then we can show the following lemma.
\begin{Lemma}\label{ESTIMATE}
 Suppose that (G) is satisfied. Then
 \begin{enumerate}
  \item[(a)] $\widehat{L}_0, \widehat{L}_1$ define operators in $L(\Lb)$ such that, for all $\alpha' < \alpha$, one has
  \begin{align*}
   \Vert \widehat{L}_0 \Vert_{L(\Lb_{\alpha}, \Lb_{\alpha'})} &\leq \frac{m}{e(\alpha - \alpha')} + \frac{\Vert a^-\Vert_{\infty} + \Vert a^+ \Vert_{\infty}}{4e^2\left(\alpha -  \alpha'\right)^2},
 \\  \Vert \widehat{L}_1G\Vert_{L(\Lb_{\alpha}, \Lb_{\alpha'})} &\leq \frac{ \| a^- \|_{L^1} e^{\alpha} + \| a^+ \|_{L^1} }{e(\alpha - \alpha')}.
 \end{align*}
 \item[(b)] $L_0^{\Delta}, L_1^{\Delta}$ define operators in $L(\K)$ such that, for all $\alpha' < \alpha$, one has 
 \[
  \| \widehat{L}_0 \|_{L(\Lb_{\alpha}, \Lb_{\alpha'})} = \| L_0^{\Delta} \|_{L(\K_{\alpha}, \K_{\alpha})}, \qquad
 \| \widehat{L}_1 \|_{L(\Lb_{\alpha}, \Lb_{\alpha'})} = \| L_1^{\Delta} \|_{L(\K_{\alpha}, \K_{\alpha})}.
 \]
 \end{enumerate}
\end{Lemma}
\begin{proof}
 Assertion (a) follows by a direct computation, see, e.g., \cite{KK18} and the references therein.
 Assertion (b) is a consequence of the duality $\Lb_{\alpha}^* = \K_{\alpha}$.
\end{proof}
 The following is due to \cite{KK18}.
\begin{Theorem}\label{LINEARTH:12}
 Let $\alpha_0 \in \R$ and $\mu_0$ be a probability measure on $\Gamma$ having correlation function $k_0 \in \K_{\alpha_0}$. Then there exists a unique
 family $(k_t)_{t \geq 0} \subset \bigcup_{\alpha \in \R}\K_{\alpha}$ with the following properties:
 \begin{enumerate}
  \item[(a)] For each $T > 0$ there exists $\alpha_T \geq \alpha_0$ such that $k_t \in \K_{\alpha_T}$, $t \in [0,T]$ and 
   $(k_t)_{t \in [0,T]}$ is the unique classical solution in $\K_{\alpha_T}$ to 
   \begin{align}\label{CORR}
    \frac{\partial k_{t}}{\partial t} = L^{\Delta}k_{t}, \ \ k_{t}|_{t=0} = k_{0}.
   \end{align}
  \item[(b)] There exists a unique family of Borel probability measures $(\mu_t)_{t \geq 0}$ on $\Gamma$ 
  such that, for any $t \geq 0$, $k_t$ given by (a) is the correlation function of $\mu_t$.
 \end{enumerate}
\end{Theorem}
Note that existence and uniqueness is only established for classical solutions to \eqref{CORR}.
Although an evolution of states $(\mu_t)_{t \geq 0}$ was constructed, its relation to \eqref{FPE} was not considered there. 
Below we show that the results obtained in Section 3 can be applied to prove the following.
\begin{Theorem}
 Suppose that condition (G) is satisfied. Then
 \begin{enumerate}
  \item[(a)] The family $(\mu_t)_{t \geq 0}$, constructed in Theorem \ref{LINEARTH:12}, is a weak solution to \eqref{FPE}.
  \item[(b)] Let $(\nu_t)_{t \geq 0}$ be another weak solution to \eqref{FPE} which admits a sequence of correlation functions $(k_{\nu_t})_{t \geq 0}$
  and suppose that for any $T > 0$ there exists $\beta_T \geq \alpha_0$ with $\sup_{t \in [0,T]} \| k_{\nu_t} \|_{\K_{\beta_T}} < \infty$.
  Then $\mu_t = \nu_t$ for all $t \geq 0$.
 \end{enumerate}
\end{Theorem}
\begin{proof}
 \textit{(a)} Let $(k_{t})_{t \geq 0}$ be the family of correlation functions corresponding to $(\mu_t)_{t \geq 0}$  given by Theorem \ref{LINEARTH:12}. 
 Fix any $T > 0$ and let $\alpha_T \geq \alpha_0$ be such that $k_{t} \in \K_{\alpha_T}$.
 Take $F \in \mathcal{FP}(\Gamma)$ and let $G \in B_{bs}(\Gamma_0)$ be such that $F = KG$. Then
 \[
  \int \limits_{\Gamma}|F(\gamma)|d\mu_t(\gamma) \leq \int \limits_{\Gamma}\sum \limits_{\eta \Subset \gamma}|G(\eta)|d\mu(\gamma)
 = \int \limits_{\Gamma_0}|G(\eta)|k_{t}(\eta) d\lambda(\eta) < \infty.
 \]
 Since, for $\alpha'' < \alpha$, one has
 $B_{bs}(\Gamma_0) \subset \Lb_{\alpha''}$ we obtain
 $\widehat{L}G \in \Lb_{\alpha}$ and hence
 \begin{align*}
  \int \limits_{\Gamma_0}|\widehat{L}G(\eta)|k_{t}(\eta)d\lambda(\eta)
  &\leq \sup \limits_{t \in [0,T]}\| k_{t}\|_{\K_{\alpha_T}} \int \limits_{\Gamma_0}|\widehat{L}G(\eta)|e^{\alpha |\eta|}d\lambda(\eta) < \infty,
 \end{align*}
 where we have used that $[0,T]\ni t \longmapsto k_{t} \in \K_{\alpha_T}$ is continuously differentiable and hence bounded.
 This implies that $\widehat{L}G \in L^1(\Gamma_0, k_t\lambda)$.
 Since $K$ can be uniquely extended to a bounded linear operator $K: L^1(\Gamma_0, k_{t}\lambda) \longrightarrow L^1(\Gamma,\mu_t)$,
 see \cite{KK02}, it follows that $LF = K\widehat{L}G \in L^1(\Gamma,\mu_t)$. Finally by \eqref{DEF:CORR} we obtain
 \begin{align}\label{EQ:00}
  \int \limits_{\Gamma}F(\gamma)d\mu_t(\gamma) = \int \limits_{\Gamma_0}G(\eta)k_{t}(\eta)d\lambda(\eta),
 \end{align}
 and using the definition of $\widehat{L}$ together with \eqref{DUALITY} yields
 \begin{align}\label{EQ:01}
   \int \limits_{\Gamma}(LF)(\gamma)d\mu_t(\gamma) &= \int \limits_{\Gamma_0}(\widehat{L}G)(\eta)k_{\mu_t}(\eta)d\lambda(\eta)
 = \int \limits_{\Gamma_0}G(\eta)(L^{\Delta}k_{\mu_t})(\eta)d\lambda(\eta).
 \end{align}
 Note that \eqref{DUALITY} was stated only for functions from $B_{bs}(\Gamma_0)$.
 However, by approximation it extends to all $k \in \K_{\alpha_T}$.
 In particular, it follows from Lemma \ref{ESTIMATE}.(b) for $\beta > \alpha_T$
 \begin{align*}
  \left| \int \limits_{\Gamma}LF(\gamma)d\mu_t(\gamma)\right| 
  &\leq \int \limits_{\Gamma_0}|G(\eta)| |L^{\Delta}k_t(\eta)|d\lambda(\eta)
 \\ &\leq \| L^{\Delta}\|_{L(\K_{\alpha_T}, \K_{\beta})} \sup \limits_{t \in [0,T]}\| k_t \|_{\K_{\alpha_T}} \int \limits_{\Gamma_0}|G(\eta)| e^{\beta |\eta|} d\lambda(\eta) < \infty.
 \end{align*}
 Since $T > 0$ was arbitrary, we conclude that $(\mu_t)_t$ is a solution to \eqref{FPE}.

 \textit{(b)} Conversely, let $(\nu_t)_{t \geq 0}$ be a solution to \eqref{FPE} with the desired properties. 
 Using \eqref{EQ:00} and \eqref{EQ:01} in this particular case shows that $k_{\nu_t}$ satisfies \eqref{WCORR}.
 In particular, $t \longmapsto \langle G, k_{\nu_t}\rangle$ is continuous for any $G \in B_{bs}(\Gamma_0)$.
 By approximation and since $\sup_{t \in [0,T]} \| k_{\nu_t}\|_{\K_{\beta_T}} < \infty$, we see that 
 $t \longmapsto \langle G, k_{\nu_t}\rangle$ is also continuous for any $G \in \Lb_{\beta_T}$.
 Thus it suffices to show that Theorem \ref{LINEARTH:11} is applicable.
 Indeed, write $\widehat{L} = \widehat{L}_{0,b} + \widehat{L}_{1,b}$, where 
 \[
  \widehat{L}_{0,b}G := \widehat{L}_0G - b|\eta|G, \qquad \widehat{L}_{1,b}G := \widehat{L}_1G + b|\eta|G.
 \]
 For $\beta > \ln(\vartheta)$ let 
$\mathcal{D}_{\beta} := \{ G \in \Lb_{\beta} \ | \ M\cdot G \in \Lb_{\beta} \}$,
 where $M(\eta) = m |\eta| + \sum_{x \in \eta}\sum_{y \in \eta \backslash x}a^-(x-y)$. 
 Condition \eqref{STABILITY} implies, for any $0 \leq G \in \mathcal{D}_{\beta}$,
  \[
   \int \limits_{\Gamma_0}\sum \limits_{x \in \eta}\int \limits_{\R^d}|G(\eta \cup y)|a^+(x-y)dy e^{\beta |\eta|}d\lambda(\eta) 
   \leq \frac{e^{-\beta}}{\vartheta}\int \limits_{\Gamma_0}\left( M(\eta) + b|\eta|\right) |G(\eta)|e^{\beta|\eta|}d \lambda(\eta).
  \]
  By \cite{TV06} and \cite{AR91} it follows that $(\widehat{L}_{0,b}, \mathcal{D}_{\beta})$ is the generator of a positive, analytic semigroup 
  $(S_{\beta}(t))_{t \geq 0}$ of contractions on $\Lb_{\beta}$. It is not difficult to see that $(S_{\beta}(t))_{t \geq 0}$ with generator 
 satisfies Remark \ref{TIMEHOMOGENEOUS} for $\beta > \ln(\vartheta)$, see, e.g., \cite{FK16} for a more general statement.
 Moreover, for any $\beta' < \beta$ and $G \in \Lb_{\beta}$, we obtain
  \[
   \Vert \widehat{L}_{1,b}G\Vert_{\Lb_{\beta'}} \leq \frac{\langle a^- \rangle e^{\alpha} + b + \langle a^+ \rangle }{e(\alpha - \alpha')}\Vert G \Vert_{\Lb_{\alpha}}.
  \]
  This shows that Theorem \ref{THEOREM:00} is applicable in the time-homogeneous case 
  with $\alpha_* = \ln(\vartheta)$, $\E_{\alpha} = \Lb_{\alpha}$, $A = L_{0,b}$ and $L_{1,b} = B$.
  Next, using Theorem \ref{LINEARTH:11} with $\mathcal{Y} = B_{bs}(\Gamma_0)$ gives the assertion.
 
\end{proof}

\subsection*{Acknowledgments}
Financial support through CRC701, project A5, at Bielefeld University is gratefully acknowledged.
The author would like to thank the anonymous referee 
for his patience while reading the first version of this work.
In particular, he gratefully acknowledges his helpful comments which lead to an improvement of 
the presentation of this work.

\begin{footnotesize}

\bibliographystyle{alpha}
\bibliography{Bibliography}

\newcommand{\etalchar}[1]{$^{#1}$}
\begin{thebibliography}{BKKK13}

\bibitem[AKR98a]{AKR98}
Sergio Albeverio, Yuri Kondratiev, and Michael R{\"o}ckner.
\newblock Analysis and geometry on configuration spaces.
\newblock {\em J. Funct. Anal.}, 154(2):444--500, 1998.

\bibitem[AKR98b]{AKR98b}
Sergio Albeverio, Yuri Kondratiev, and Michael R{\"o}ckner.
\newblock Analysis and geometry on configuration spaces: the {G}ibbsian case.
\newblock {\em J. Funct. Anal.}, 157(1):242--291, 1998.

\bibitem[AR91]{AR91}
Wolfgang Arendt and Abdelaziz Rhandi.
\newblock Perturbation of positive semigroups.
\newblock {\em Arch. Math. (Basel)}, 56(2):107--119, 1991.

\bibitem[BHP15]{BHP15}
Rafael~F. Barostichi, A.~Alexandrou Himonas, and Gerson Petronilho.
\newblock A {C}auchy-{K}ovalevsky theorem for nonlinear and nonlocal equations.
\newblock In {\em Analysis and geometry}, volume 127 of {\em Springer Proc.
  Math. Stat.}, pages 59--68. Springer, Cham, 2015.

\bibitem[BKK15]{BKK15}
Christoph Berns, Yuri Kondratiev, and Oleksandr Kutoviy.
\newblock Markov jump dynamics with additive intensities in continuum: state
  evolution and mesoscopic scaling.
\newblock {\em J. Stat. Phys.}, 161(4):876--901, 2015.

\bibitem[BKKK13]{BKKK13}
Christoph Berns, Yuri Kondratiev, Yuri Kozitsky, and Oleksandr Kutoviy.
\newblock Kawasaki dynamics in continuum: micro- and mesoscopic descriptions.
\newblock {\em J. Dynam. Differential Equations}, 25(4):1027--1056, 2013.

\bibitem[Cap02]{CAPS}
Oliver Caps.
\newblock {\em Evolution equations in scales of {B}anach spaces}, volume 140 of
  {\em Teubner-Texte zur Mathematik [Teubner Texts in Mathematics]}.
\newblock B. G. Teubner, Stuttgart, 2002.

\bibitem[FFH{\etalchar{+}}15]{FFHKKK15}
Dmitri Finkelshtein, Martin Friesen, Haralambos Hatzikirou, Yuri Kondratiev,
  Tyll Kr\"uger, and Oleksandr Kutoviy.
\newblock Stochastic models of tumour development and related mesoscopic
  equations.
\newblock {\em Inter. Stud. Comp. Sys.}, 7:5--85, 2015.

\bibitem[Fin11]{F11}
Dimitri Finkelshtein.
\newblock Functional evolutions for homogeneous stationary death-immigration
  spatial dynamics.
\newblock {\em Methods Funct. Anal. Topology}, 17(4):300--318, 2011.

\bibitem[Fin15]{F15}
Dimitri Finkelshtein.
\newblock Around {O}vsyannikov's method.
\newblock {\em Methods Funct. Anal. Topology}, 21(2):134--150, 2015.

\bibitem[FK16]{FK16}
Martin Friesen and Oleksandr Kutoviy.
\newblock Evolution of states and mesoscopic scaling for two-component
  birth-and-death dynamics in continuum.
\newblock {\em Methods Funct. Anal. Topology}, 22(4):346--374, 2016.

\bibitem[FK18a]{FK18b}
Martin Friesen and Yuri Kondratiev.
\newblock Stochastic averaging principle for spatial birth-and-death evolutions
  in the continuum.
\newblock {\em J. Stat. Phys.}, 171(5):842--877, 2018.

\bibitem[FK18b]{FK18}
Martin Friesen and Oleksandr Kutoviy.
\newblock Nonlinear perturbations of evolution systems in scales of {B}]anach
  spaces.
\newblock {\em arXiv:1805.10597 [math.FA]}, 2018.

\bibitem[FKK09]{FKK09}
Dimiri Finkelshtein, Yuri Kondratiev, and Oleksandr Kutoviy.
\newblock Individual based model with competition in spatial ecology.
\newblock {\em SIAM J. Math. Anal.}, 41(1):297--317, 2009.

\bibitem[FKK12]{FKK12}
Dimitri Finkelshtein, Yuri Kondratiev, and Oleksandr Kutoviy.
\newblock Semigroup approach to birth-and-death stochastic dynamics in
  continuum.
\newblock {\em J. Funct. Anal.}, 262(3):1274--1308, 2012.

\bibitem[FKK13a]{FKK13}
Dimitri Finkelshtein, Yuri Kondratiev, and Yuri Kozitsky.
\newblock Glauber dynamics in continuum: a constructive approach to evolution
  of states.
\newblock {\em Discrete Contin. Dyn. Syst.}, 33(4):1431--1450, 2013.

\bibitem[FKK13b]{FKK13EST}
Dimitri Finkelshtein, Yuri Kondratiev, and Oleksandr Kutoviy.
\newblock Establishment and fecundity in spatial ecological models: statistical
  approach and kinetic equations.
\newblock {\em Infin. Dimens. Anal. Quantum Probab. Relat. Top.},
  16(2):1350014, 24, 2013.

\bibitem[FKKO15]{FKK15WRMODEL}
Dmitri Finkelshtein, Yuri Kondratiev, Oleksandr Kutoviy, and Maria Jo\~ao
  Oliveira.
\newblock Dynamical {W}idom-{R}owlinson model and its mesoscopic limit.
\newblock {\em J. Stat. Phys.}, 158(1):57--86, 2015.

\bibitem[FKKZ14]{FKKZ14}
Dmitri Finkelshtein, Yuri Kondratiev, Oleksandr Kutoviy, and Elena Zhizhina.
\newblock On an aggregation in birth-and-death stochastic dynamics.
\newblock {\em Nonlinearity}, 27(6):1105--1133, 2014.

\bibitem[FKO09]{FKO07}
Dmitri Finkelshtein, Yuri Kondratiev, and Maria Jo\~ao Oliveira.
\newblock Markov evolutions and hierarchical equations in the continuum. {I}.
  {O}ne-component systems.
\newblock {\em J. Evol. Equ.}, 9(2):197--233, 2009.

\bibitem[Fri17]{F17}
Martin Friesen.
\newblock Non-equilibrium dynamics for a {W}idom-{R}owlinson type model with
  mutations.
\newblock {\em J. Stat. Phys.}, 166(2):317--353, 2017.

\bibitem[GK06]{GK06}
Nancy Garcia and Thomas Kurtz.
\newblock Spatial birth and death processes as solutions of stochastic
  equations.
\newblock {\em ALEA Lat. Am. J. Probab. Math. Stat.}, 1:281--303, 2006.

\bibitem[Hq13]{H13}
Hern\'an~R. Henr\'\i~quez.
\newblock Existence of solutions of the nonautonomous abstract {C}auchy problem
  of second order.
\newblock {\em Semigroup Forum}, 87(2):277--297, 2013.

\bibitem[KK02]{KK02}
Yuri Kondratiev and Tobias Kuna.
\newblock Harmonic analysis on configuration space. {I}. {G}eneral theory.
\newblock {\em Infin. Dimens. Anal. Quantum Probab. Relat. Top.},
  5(2):201--233, 2002.

\bibitem[KK06]{KK06}
Yuri Kondratiev and Oleksandr Kutoviy.
\newblock On the metrical properties of the configuration space.
\newblock {\em Math. Nachr.}, 279(7):774--783, 2006.

\bibitem[KK18]{KK18}
Yuri Kondratiev and Yuri Kozitsky.
\newblock The evolution of states in a spatial population model.
\newblock {\em J. Dynam. Differential Equations}, 30(1):135--173, 2018.

\bibitem[KKM08]{KKM08}
Yuri Kondratiev, Oleksandr Kutoviy, and Robert Minlos.
\newblock On non-equilibrium stochastic dynamics for interacting particle
  systems in continuum.
\newblock {\em J. Funct. Anal.}, 255(1):200--227, 2008.

\bibitem[KKP08]{KKP08}
Yuri Kondratiev, Oleksandr Kutoviy, and Sergey Pirogov.
\newblock Correlation functions and invariant measures in continuous contact
  model.
\newblock {\em Infin. Dimens. Anal. Quantum Probab. Relat. Top.},
  11(2):231--258, 2008.

\bibitem[KL05]{KL05}
Yuri Kondratiev and Eugene Lytvynov.
\newblock Glauber dynamics of continuous particle systems.
\newblock {\em Ann. Inst. H. Poincar\'e Probab. Statist.}, 41(4):685--702,
  2005.

\bibitem[Kol13]{KOL13}
Vassili Kolokoltsov.
\newblock Nonlinear {L}\'evy and nonlinear {F}eller processes: an analytic
  introduction.
\newblock In {\em Mathematics and life sciences}, volume~1 of {\em De Gruyter
  Ser. Math. Life Sci.}, pages 45--69. De Gruyter, Berlin, 2013.

\bibitem[Paz83]{PAZ83}
Amnon Pazy.
\newblock {\em Semigroups of linear operators and applications to partial
  differential equations}, volume~44 of {\em Applied mathematical sciences ;
  44}.
\newblock Springer, New York [u.a.], 2., corr. print edition, 1983.

\bibitem[Saf95]{SAF95}
Mikhail~V. Safonov.
\newblock The abstract {C}auchy-{K}ovalevskaya theorem in a weighted {B}anach
  space.
\newblock {\em Comm. Pure Appl. Math.}, 48(6):629--637, 1995.

\bibitem[Tg08]{T08}
Feride Ti\u~glay.
\newblock The {C}auchy problem and integrability of a modified
  {E}uler-{P}oisson equation.
\newblock {\em Trans. Amer. Math. Soc.}, 360(4):1861--1877, 2008.

\bibitem[TV06]{TV06}
Horst Thieme and J\"urgen Voigt.
\newblock Stochastic semigroups: their construction by perturbation and
  approximation.
\newblock In {\em Positivity {IV}---theory and applications}, pages 135--146.
  Tech. Univ. Dresden, Dresden, 2006.

\bibitem[WZ02]{WZ02}
Liming Wu and Yiping Zhang.
\newblock Existence and uniqueness of {$C_0$}-semigroup in {$L^\infty$}: a new
  topological approach.
\newblock {\em C. R. Math. Acad. Sci. Paris}, 334(8):699--704, 2002.

\bibitem[WZ06]{WZ06}
Liming Wu and Yiping Zhang.
\newblock A new topological approach to the {$L^\infty$}-uniqueness of
  operators and the {$L^1$}-uniqueness of {F}okker-{P}lanck equations.
\newblock {\em J. Funct. Anal.}, 241(2):557--610, 2006.

\end{thebibliography}

\end{footnotesize}

\end{document}